\documentclass[12pt]{article}

\usepackage[left=1.2in, right=1.2in, top=1in, bottom=1in]{geometry}

\usepackage{amsmath,amsthm,amsfonts,amssymb,mathrsfs,mathtools}
\usepackage{graphicx} 
\usepackage{xcolor}
\usepackage{float}
\usepackage[normalem]{ulem}
\usepackage[font=footnotesize,labelfont=bf]{caption}
\usepackage{authblk}

\usepackage{hyperref}
\hypersetup{
    colorlinks=true, 
    linktoc=all, 
    linkcolor=black, 
    citecolor=black 
}
\usepackage[nocompress]{cite} 

\usepackage{mathtools}
\mathtoolsset{showonlyrefs=false} 

\newcommand{\R}{\mathbb{R}} 
\newcommand{\N}{\mathbb{N}} 
\renewcommand{\S}{\mathbb{S}} 
\newcommand{\C}{\mathbb{C}} 
\newcommand{\eps}{\varepsilon} 
\newcommand{\p}{\partial}
\newcommand{\diff}{\,\mathrm{d}} 
\newcommand{\radonR}{\mathscr{R}}
\newcommand{\e}{\mathrm{e}}
\renewcommand{\i}{\mathrm{i}}
\newcommand{\FT}{\mathscr{F}} 
\newcommand{\norm}[2]{\Vert #1 \Vert_{#2}}

\newtheorem{thm}{Theorem}
\newtheorem{lemma}{Lemma}
\theoremstyle{definition}
\newtheorem{example}{Example}
\theoremstyle{remark}
\newtheorem{remark}{Remark}

\title{X-ray imaging from nonlinear waves: numerical reconstruction of a cubic nonlinearity}
\author{Suvi Anttila, Markus Harju, and Teemu Tyni\\Research Unit of Applied and Computational Mathematics, University of Oulu, Finland}

\date{}

\begin{document}

\maketitle

\begin{abstract}
    We study an inverse boundary value problem for the nonlinear wave equation in $2 + 1$ dimensions. The objective is to recover an unknown potential $q(x, t)$ from the associated Dirichlet-to-Neumann map using real-valued waves.
    We propose a direct numerical reconstruction method for the Radon transform of $q$, which can then be inverted using standard X-ray tomography techniques to determine $q$. Our implementation introduces a spectral regularization procedure to stabilize the numerical differentiation step required in the reconstruction, improving robustness with respect to noise in the boundary data. We give rigorous justification and optimal stability estimates for the regularized spectral differentiation of noisy measurements, which may be of independent interest.
    Numerical experiments demonstrate the feasibility of recovering potentials from boundary measurements of nonlinear waves and illustrate the advantages of the Radon-based reconstruction. 
\end{abstract}


\section{Introduction}

We consider an inverse boundary value problem for the nonlinear wave equation
\begin{equation}\label{eq: nonlinear wave}
\begin{cases}
\partial_t^2 u(x,t) - \Delta u(x,t) + q(x,t) u(x,t)^p= 0, &(x,t) \in \Omega \times [0,T],\\
u(x,t) = f(x,t), &(x,t)\in \partial \Omega \times [0,T],\\
u(x,0) = \partial_t u(x,0) = 0, &x \in \Omega
\end{cases}
\end{equation}
in a bounded Lipschitz domain $\Omega\subset \R^2$, where $p \geq 2$ is an integer, and $q\in C^{s}(\Omega \times [0, T])$ for a given $s > 1$.
The inverse problem we study is the recovery of the unknown potential $q$ given the Dirichlet-to-Neumann map (DN-map)
\begin{align*}
    \Lambda_q : H^{s+1}(\partial \Omega \times [0, T]) \to H^{s}(\partial \Omega \times [0, T]), \hspace{0.3cm} \Lambda_q(f) = \partial_{\nu} u|_{\partial \Omega \times [0,T]},
\end{align*}
where $u$ satisfies \eqref{eq: nonlinear wave} and $\nu$ is the outward pointing normal of the lateral boundary 
\begin{equation*}
    \Sigma := \partial \Omega \times [0, T].
\end{equation*}
Here and throughout, $H^{s}$ denotes the standard $L^2$-based Sobolev space. 
In this work, we develop a numerical method for reconstructing the Radon transform $\radonR(q)$ of $q$, often referred to in practice as the sinogram, from boundary measurements of nonlinear waves.
The potential $q$ is then obtained from $\radonR(q)$ using standard X-ray computed tomography (CT) reconstruction methods. Our numerical approach follows the theory developed in \cite{LLPT20}.

In recent years, many new results related to inverse problems for nonlinear equations have been published. This is largely due to the property that nonlinear wave interactions can be used as a beneficial tool in solving the inverse problem, proved first by Kurylev, Lassas and Uhlmann in \cite{kurylev_inverse_2018}. Their work focused on the scalar wave equation with a quadratic nonlinearity and they showed that local measurements of solutions for the wave equation determine the global topology, differentiable structure, and the conformal class of the metric $g$ on a globally hyperbolic four-dimensional Lorentzian manifold. Their method, now called \emph{higher order linearization}, has made inverse problems for nonlinear models more approachable and allowed the solving of inverse problems for nonlinear equations for which the corresponding linear problems are still unsolved.

Following \cite{kurylev_inverse_2018}, inverse problems for general semi-linear wave equations on Lorentzian manifolds were considered in \cite{lassas_inverse_2018}, while the analogous problem for the Einstein--Maxwell equations was investigated in \cite{lassas_determination_2017}. Since these works, higher order linearization has been applied extensively to inverse problems for nonlinear models.
The method has been used to study, for example, elliptic equations \cite{carstea_reconstruction_2019, feizmohammadi_inverse_2020, lassas_partial_2020, krupchyk_partial_2021, krupchyk_remark_2020, lassas_inverse_2021, liimatainen_uniqueness_2024}, real principal type equations \cite{oksanen_inverse_2024}, nonlinear elastic wave equations \cite{de_hoop_nonlinear_2019, de_hoop_nonlinear_2020}, and nonlinear wave equations on Lorentzian manifolds \cite{hintz_dirichlet--neumann_2022, hintz_inverse_2022, lassas_stability_2025, wang_inverse_2019, feizmohammadi_recovery_2022}. The method has also been used for inverse problems in the context of the Boltzmann equation \cite{balehowsky_inverse_2022, lai_reconstruction_2021}, the Einstein equations \cite{kurylev_inverse_2022, uhlmann_determination_2020}, and the Yang-Mills equations \cite{chen_detection_2021, chen_inverse_2021}. Numerical studies on the subject have been conducted, for example, in the context of the Westervelt equation (ultrasound imaging) in \cite{acosta_nonlinear_2022, kaltenbacher_identification_2021}, for nonlinear waves in \cite{sa_barreto_recovery_2022, LLPT24}, and in the context of Helmholtz scattering in \cite{griesmaier_inverse_2022}.

The works \cite{LLPT20, LLPT24} are closely related to ours and form the basis that the present work builds on. The stability and uniqueness of the inverse problem of equation \eqref{eq: nonlinear wave}, in a more general case where $\Omega \subset \R^n$, was studied in \cite{LLPT20}. In particular, for $n \geq 2$, it was shown that the DN-map $\Lambda_q$ uniquely determines $\radonR(q)$ and a stable reconstruction algorithm was provided. The work was primarily analytical, and did not address the challenges of numerical implementation. A first step toward computation was taken in \cite{LLPT24}, where a numerical implementation was developed in the one-dimensional domain $\Omega \subset \R$ for a pointwise reconstruction of the potential $q$. While this provides useful insight for our work, the one-dimensional method does not directly extend to higher dimensions, where the reconstruction method involves recovering the Radon transform of the potential $q$.

In a different direction, \cite{sa_barreto_recovery_2022} considers a kind of near-field scattering problem for a wave-equation with a cubic nonlinearity. There, a domain $\Omega$ is probed using complex-valued harmonic waves from outside and the transmitted waves are measured, resulting in the recovery of the Radon transform of $q$. This approach requires a different measurement setup to ours.

In contrast to the works mentioned above, we propose the first numerical implementation of the reconstruction method from \cite{LLPT20} in two space-dimensions. The novelty of this work is the implementation of a direct, non-iterative method that uses real-valued waves and produces a Radon transform of the potential $q$ using waves, instead of X-rays. To mitigate the effects of noise in the measurement data, we introduce a spectral regularization step for evaluating certain numerical derivatives involved in our approach. We give rigorous justification and optimal stability estimates of the H\"older type
\begin{equation*}
    \norm{f^{(p)} - D^{(p)}_R f_\delta}{L^2} \asymp \norm{f}{H^s}^{p/s} \norm{f - f_{\delta}}{L^2}^{1-p/s}
\end{equation*}
for the regularized spectral differentiation $D^{(p)}_R f_\delta$ of noisy measurements $f_\delta\in L^2$ of $f\in H^s$, which may be of independent interest. Finally, for comparison, we also implement a method for pointwise reconstruction of $q$, which does not resort to the Radon transform.

The structure of this paper is as follows. In Section~\ref{sec: rec methods} we discuss the theoretical reconstruction methods, following mainly the theory developed in \cite{LLPT20}. Section~\ref{sec: numerical implem} is focused on different aspects of the numerical implementations we have carried out for this work. We discuss the numerical solution of the forward problem of \eqref{eq: nonlinear wave} and the construction of the synthetic DN-map in Section~\ref{sec: forward model and dn map}, and the numerical implementation of the reconstruction algorithm for $\radonR(q)$ in Section \ref{sec: numerical implem of the ip}. In Section~\ref{sec: numerical differentiation}, we discuss a regularized spectral differentiation method used to approximate the numerical derivatives needed in the reconstruction algorithm. We derive stability estimates for the regularized differentiation. Finally, in Section~\ref{sec: numerical examples} we present and compare examples of different potentials reconstructed both by the Radon transform reconstruction method and the pointwise reconstruction method.

\section{Reconstruction methods}\label{sec: rec methods}

The reconstruction of $q$ is based on a higher-order linearization in small parameters in the boundary values $f$: Given suitable boundary data $f$ and $\Lambda_q(f)$ for \eqref{eq: nonlinear wave}, certain higher derivatives of $\Lambda_q$ with respect to $f$ allow us to reconstruct $q$ or $\radonR(q)$, depending on the choice of $f$.

\subsection{Higher order linearization method}

Let $\vec{\eps} = (\eps_1, \ldots, \eps_m)$ be small parameters and $f_1, \ldots, f_m \in H^{s+1}(\Sigma)$, where $s \in \N$ and $s \geq 1$.  Let $u_{\vec{\eps}}\in \cap_{0\leq k\leq s+1} C^k([0,T];H^{s+1-k}(\Omega))$ be the (unique small) solution to
\begin{equation}\label{eq: nonlinear wave eps: general case}
    \begin{cases}
        \square u_{\vec{\eps}}(x,t) = -q(x,t)u_{\vec{\eps}}^p(x,t), &(x,t) \in \Omega \times [0,T],\\
        u_{\vec{\eps}}(x,t) = \sum_{j = 1}^m \eps_j f_j(x,t), &(x,t)\in \Sigma,\\
        u_{\vec{\eps}}(x,0)=\partial_t u_{\vec{\eps}}(x,0) = 0, &x\in\Omega.
    \end{cases}
\end{equation}
Here we denote the wave operator by $\square := \partial^2_t - \Delta$.

For \emph{higher order linearization}, we will differentiate equation \eqref{eq: nonlinear wave eps: general case} $p$ times with respect to $\vec{\eps}$. For this, we denote
\begin{equation*}
    w := \partial^{\sigma}_{\vec{\eps}} u_{\vec{\eps}}|_{\vec{\eps} = \vec{0}},
\end{equation*}
where $\sigma = (\sigma_1, \ldots, \sigma_m)$ is a multi-index with $|\sigma| = p$. By differentiating the equation \eqref{eq: nonlinear wave eps: general case}, we see that
$w$, the $p$th linearization of $u_{\vec{\eps}}$,
satisfies a linear wave equation of the form
\begin{equation}\label{eq: wave w}
    \begin{cases}
        \square w = -p! q \prod_{i = 1}^m v_i^{\sigma_i}, & \text{in } \Omega \times [0,T],\\
        w = 0, & \text{on } \Sigma,\\
        w|_{t = 0} = \partial_t w|_{t = 0} = 0, &  \text{in } \Omega, 
    \end{cases}
\end{equation}
where, in turn, the functions $v_j:= \partial_{\eps_j} u_{\vec{\eps}}|_{\eps_j = 0}$ solve
\begin{equation}\label{eq: wave v_j}
    \begin{cases}
        \square v_j = 0, & \text{in}\, \Omega \times [0,T], \\
        v_j = f_j, & \text{on}\, \Sigma, \\
        v_j|_{t = 0} = \partial_t v_j|_{t = 0} = 0, &\text{in}\, \Omega
    \end{cases}
\end{equation}
for $j = 1, \ldots, m$.
In this manner, we obtain linear wave equations from the nonlinear wave equation \eqref{eq: nonlinear wave eps: general case}.

Note that because the DN-map
is given, also the normal derivative of $w$ on the lateral boundary is known, because
\begin{equation*}
    \partial_{\nu} w = \partial_{\nu}(\partial^{\sigma}_{\vec{\eps}} u_{\vec{\eps}}|_{\vec{\eps} = \vec{0}}) = \partial^{\sigma}_{\vec{\eps}} \big(\Lambda_q\bigg(\sum_{j = 1}^m \eps_j f_j\bigg)\big)\big|_{\vec{\eps} = \vec{0}}\quad \text{on } \Sigma.
\end{equation*}

Next, we define an auxiliary function $v_0$ to compensate for the fact that $\partial_{\nu} w$ is unknown at $t = T$. Let $f_0 \in H^{s+1}(\Sigma)$ and find a $v_0$ satisfying
\begin{equation}\label{eq: wave v_0}
    \begin{cases}
        \square v_0 = 0, & \text{in}\,\Omega \times [0,T], \\
        v_0 = f_0, & \text{on}\, \Sigma, \\
        v_0|_{t = T} = \partial_t v_0|_{t = T} = 0, &\text{in}\, \Omega.
    \end{cases}
\end{equation}
Multiplying equation \eqref{eq: wave w} by $v_0$, integrating over $\Omega \times [0, T]$, and integrating by parts, we arrive at the integral identity
\begin{align}\label{eq: integralidentity: eps general case}
    \begin{split}
        -p!\int_{\Omega\times [0,T]} q v_0 \prod_{j = 1}^m v_j^{\sigma_j}  \diff x  \diff t &= \int_{\Omega \times [0,T]} v_0 \square w \diff x \diff t \\
        &= - \int_{\Sigma}v_0 \partial_{\vec{\eps}}^{\sigma}\big|_{\vec{\eps}=\vec{0}}\Lambda_q \bigg(\sum_{j = 1}^m \eps_j f_j\bigg) \diff S(x),
    \end{split}
\end{align}
where $\diff S(x)$ is the Lebesgue surface measure on $\Sigma$. For justification of the differentiability of the DN-map $\Lambda_q$, see~\cite{LLPT20}.

Everything on the right-hand side of equation \eqref{eq: integralidentity: eps general case} is known if we know $\Lambda_q$, and thus the integral
\begin{align}\label{eq: integral q_generalvj:s}
    \int_{\Omega\times [0,T]} q v_0 \prod_{j = 1}^m v_j^{\sigma_j}  \diff x  \diff t
\end{align}
is determined by the DN-map $\Lambda_q$. Recall that each $v_j, j = 1, \ldots, m$, satisfies the linear wave equation \eqref{eq: wave v_j}. Since the boundary value of the nonlinear wave equation \eqref{eq: nonlinear wave eps: general case} is $\sum_{j=1}^m \eps_j f_j$, we may freely choose each $f_j$, and consequently $v_j$, to extract the potential $q$ from the integral \eqref{eq: integral q_generalvj:s}. This approach is H\"older stable, see~\cite{LLPT20}.

\subsection{Selecting the boundary values}\label{section: choosing boundary}

We divide the discussion into two parts. First, we study a method that produces a reconstruction of the Radon transform $\radonR(q)$ and then a method that aims to reconstruct the potential $q$ pointwise.

\subsubsection{Reconstruction of the Radon transform of $q$}\label{sec: choosing boundary for radon}

For $\Omega\subset \R^n$ we define the \emph{partial Radon transform} of a function $G=G(x,t)\in C^\infty_0(\Omega\times\R)$, in its spatial variable $x$ as 
\begin{equation}\label{def:Radon_transform}
 \radonR(G)(t,\theta,\eta)=\int_{x\cdot \theta =\eta} G(x,t) \diff S(x), \quad \theta\in \S^{n-1}, \ \eta\in \R.
\end{equation}
(For more details about the Radon transform, see \cite{helgason1999radonBook}.) Here and throughout, $\S^{n-1}=\{x\in\R^n: |x|=1\}$ denotes the unit sphere.

Let
\[
r:=\inf\{r>0: \Omega\subset B_r(x),\text{ for some } x\in\R^n\}.
\]
Then there exists an $x_0\in\R^n$ (which we assume without loss of generality to be the origin) such that $\Omega\subset B_r(x_0)$.
We define the admissible reconstruction domain by setting $T>4r$, $t_1>2r$, $t_2<T-2r$, and defining
\begin{equation}\label{eq: W pointwise}
    W = \Omega\times[t_1,t_2],
\end{equation}
which in Figure~\ref{fig: rec Radon waves} (left) lies within the causal domain of $[0,T]\times\p\Omega$. The corresponding admissible domain of the Radon transform is then
\begin{equation}
    W_{\text{Radon}}:=[t_1,t_2]\times \S^{n-1} \times [-r,r].
\end{equation}

Consider the recovery of the partial Radon transform \eqref{def:Radon_transform} of $q$ for a parameter tuple $(t_0, \theta_0, \eta_0)$ in the admissible reconstruction set $W_{\text{Radon}}$. The Radon transform is a line integral of $q$ along
\[
L(t_0, \theta_0, \eta_0) := \{(x, t) \in \Omega \times [0, T] \mid x \cdot \theta_0 = \eta_0, t = t_0\}.
\]
To reconstruct $\radonR(q)(t_0, \theta_0, \eta_0)$, we choose the waves $v_j, j = 0, \ldots, m$ so that their product $v_0 \prod_{j = 1}^m v_j^{\sigma_j}$ approximates the delta distribution on $L(t_0, \theta_0, \eta_0)$. Then the integral \eqref{eq: integral q_generalvj:s}, up to a constant factor, approximates $\radonR(q)(t_0, \theta_0, \eta_0)$. In three dimensions (two spatial dimensions plus time), the required delta concentration on a line is obtained by focusing two plane waves that intersect along $L(t_0, \theta_0, \eta_0)$.

By \cite[Lemma 8]{LLPT20} we have the following result.
\begin{lemma}\label{lemma:approx Radon}
Let $G$ be compactly supported and Lipschitz. Let $t_0\in \R$ and $\tau>0$. There exists $C>0$ (depending only on $\mathrm{supp} (G)$) such that  
\begin{equation*}
\left| \radonR(G)(t_0,\theta, \eta)-  \frac{\tau}{\pi} \int_{\R}\int_{\R^{n}}G(x,t) \e^{-\tau((x\cdot \theta-\eta)^2 + (t-t_0)^2)} \diff x \diff t \right|
 \leq \frac{\sqrt{\pi}}{2} C \left\| G \right\|_{\mathrm{Lip}}\tau^{-1/2}.
\end{equation*}
Here $\left\| G \right\|_{\mathrm{Lip}}=\inf\{ c\ge 0: |G(x)-G(y)|\le c|x-y|\}$ is a Lipschitz semi-norm and $C$ is independent of $\theta\in \S^{n-1}$ and $\eta\in \R$. In particular, the integral on the left converges uniformly to $\radonR(G)(t_0,\theta, \eta)$ when $\tau \to \infty$.
\end{lemma}

We specialize our considerations to the case $p = 3$.
We define
\begin{equation}\label{eq: H(l)}
    H(l) := \varphi_h (l) \tau^{1/2} \e^{-\frac{1}{2}\tau \, l^2},
\end{equation}
where $\varphi_h \in C^{\infty}_0(\mathbb{R})$ is a cutoff function supported on $[-h, h]$ with $\varphi_h(0) = 1$. Using the function $H$, we define
\begin{equation}\label{eq: planewaves}
\begin{split}
    H_1^{\tau, (t_0, \theta,\eta)}(x,t) &:= H(x\cdot \theta -t - (\eta - t_0)),\\
    H_2^{\tau, (t_0, \theta,\eta)}(x,t) &:= H(- x\cdot \theta -t + (\eta + t_0)).    
\end{split}
\end{equation}
Both functions $H_1^{\tau, (t_0, \theta,\eta)}$ and $H_2^{\tau, (t_0, \theta,\eta)}$ satisfy the linear wave equation, and we regard them as plane waves propagating to direction of $\theta$ and $-\theta$, respectively.

We set 
\begin{align*}
    f_1 = \sqrt[3]{H_1^{\tau, (t_0, \theta_0, \eta_0)}}\bigg|_{\Sigma}
\end{align*}
and use $f = \eps_1 f_1$ as the boundary value of \eqref{eq: nonlinear wave eps: general case}. Therefore, according to \eqref{eq: wave v_j}, $v_1 = \sqrt[3]{H_1}$. In the higher-order linearization method, we differentiate equation \eqref{eq: nonlinear wave eps: general case} three times with respect to $\eps_1$, i.e., we choose $\sigma = 3$. In addition, we set
\begin{equation*}
    v_0|_{\Sigma} := f_0 = H_2^{\tau, (t_0, \theta_0, \eta_0)}|_{\Sigma}
\end{equation*}
as the boundary value for the auxiliary function $v_0$ in \eqref{eq: wave v_0}.

Now, the function $v_1$ is a plane wave traveling in the direction $\theta_0$ and $v_0$ is a plane wave traveling in the opposite direction. These waves meet on the line $L(t_0, \theta_0, \eta_0)$.
When $\tau$ is large enough, the product of these waves approximates the delta distribution on this line. More specifically, the product $v_0 v_1^3 \approx \tau \e^{-\tau ((x \cdot \theta_0 - \eta_0)^2 + (t-t_0)^2)}$ near the line $L(t_0, \theta_0, \eta_0)$, and by the integral identity \eqref{eq: integralidentity: eps general case} and Lemma \ref{lemma:approx Radon}, we have
\begin{equation}\label{eq: recoequ Radon}
    3! \pi \radonR(q)(t_0, \theta_0, \eta_0) = \lim_{\tau\to \infty} \int_{\Sigma} f_0 \partial^3_{\eps_1} \big|_{\eps_1 = 0} \Lambda(\eps_1 f_1) \diff S(x).
\end{equation}
Since the right-hand side of this equation is known from the knowledge of the DN-map, we have successfully formulated a reconstruction of $\radonR(q)(t_0, \theta_0, \eta_0)$. This can then be repeated for all parameter tuples in the admissible set $W_{\text{Radon}}$.

\begin{remark}
    Above, differentiating thrice with respect to one small parameter $\eps_1$ is sufficient for the recovery of $\radonR(q)$. The benefit here is that it suffices to send only \emph{a single plane wave} from each direction $\theta_0 \in \S^{n-1}$. One could also use several small parameters in the boundary values and in this way obtain different products of linear waves in the related integral identity. There are many ways in which the boundary values can be chosen, such as distorted plane waves \cite{kurylev_inverse_2018}, complex harmonic wave packets \cite{sa_barreto_recovery_2022}, or Gaussian beams \cite{lassas_stability_2025}.
\end{remark}

\begin{remark}\label{rem: limited tomo}
    We make a brief detour to limited angle X-ray tomography.
    Consider the continuous 2D Radon transform
    \[
    \radonR(g)(\theta,s) := \int_{L(\theta,s)} g(x) \diff S(x),
    \]
    where $L(\theta,s):=\{ x\in\R^2 \mid x\cdot\theta=s\}$ is the line with normal direction $\theta\in\S^1$ and $s\in\R$ is the signed distance to the origin, see~\cite{helgason1999radonBook}. Then, given the Radon transform $\radonR(g)(\theta,s)$ in an arbitrarily small neighborhood of the point $(\theta_0,s_0)\in \S^1\times \R$, the singularities of $g$ on the line $L(\theta_0,s_0)$ to direction $\theta_0$ can be stably recovered; see~\cite{QuintoRadonSingularities}.
    
    Returning to the inverse problem of determining the nonlinear potential $q$, according to \cite[Proposition 4]{LLPT20}, there is a choice of plane waves which can be used to uniquely determine the Radon (X-ray) transform $\radonR(q)$ of the potential $q$ from the DN-map. The reconstruction of $\radonR(q)$ occurs pointwise in the Radon domain of distances and incident angles. Hence, by limiting the incident angles of plane waves, it is possible to speak of \emph{limited angle X-ray tomography from nonlinear waves}, explored in Figure~\ref{img: recos L limited angles}. Roughly speaking, 
    if plane waves can be sent into the domain $\Omega$ such that the normal directions of the wavefronts are close to $\theta_0\in \S^1$ and the waves' travel times are close to $s_0$, then one can attempt to extract information about the singularities of the potential $q$ along $L(\theta_0,s_0)$ to the direction $\theta_0$. This can be useful for imaging near boundaries of domains over short times.
\end{remark}

\subsubsection{Pointwise reconstruction of $q$}

Let us briefly describe how the method changes if we want to reconstruct $q$ pointwise. In this case the waves $v_j, j = 0, \ldots, m$, are chosen so that their product approximates the delta distribution at a point $(x_0, t_0) \in W$, where the potential $q$ can be then determined.

The following Lemma is from \cite[Lemma 20]{lassas_stability_2025}.
\begin{lemma}\label{lemma:approx identity q}
Let  $\tau>0$ and let $b$ be Lipschitz. Then
\[
\left| b(z_0)-  \left(\frac{\tau}{\pi}\right)^{\frac{n}{2}} \int_{\mathbb{R}^{n}}b(z) \e^{-\tau |z-z_0|^2} \diff z \right| \leq c_n \left\|b \right\|_\mathrm{Lip}\tau^{-1/2}
\]
holds true for all $z_0\in \mathbb{R}^{n}$. In particular, the integral on the left converges uniformly to $b(z_0)$ when $\tau \to \infty$. Here $c_n:=\Gamma\left( \frac{n+1}{2}\right) / \Gamma \left( \frac{n}{2} \right)$.
\end{lemma}

Using the function $H$ defined in \eqref{eq: H(l)}, we set
\begin{align*}
    H_1^{\tau, \theta, (x_0, t_0)}(x,t) &:= H((t-t_0) - \theta \cdot (x - x_0)),\\
    H_2^{\tau, \theta, (x_0, t_0)}(x,t) &:= H((t-t_0) + \theta \cdot (x - x_0)),\\
    H_3^{\tau, \theta, (x_0, t_0)}(x,t) &:= H((t-t_0) - \theta^{\perp} \cdot (x - x_0)),
\end{align*}
where $\theta \in \S^1$ and $\theta^{\perp}\in \S^1$ is perpendicular to $\theta$. As before, the functions $H_j^{(x_0, t_0)} := H_j^{\tau, \theta, (x_0, t_0)}, j = 1, 2, 3$, are plane waves propagating to the direction $\theta, -\theta$ and $\theta^{\perp}$, respectively.

We fix an arbitrary $\theta \in \S^1$ and set
\begin{equation*}
        f_1 = H_1^{(x_0, t_0)}|_{\Sigma}, \quad
        f_2 = \sqrt{H_2^{(x_0, t_0)}}\big|_{\Sigma},
\end{equation*}
and use $u = \eps_1 f_1 + \eps_2 f_2$ as the boundary value of  \eqref{eq: nonlinear wave eps: general case}. Therefore $v_1 =  H_1^{(x_0, t_0)}$ and $v_2 =  \sqrt{H_2^{(x_0, t_0)}}$. In the higher-order linearization method, we choose $\sigma = (1, 2)$ and we set
\begin{equation*}
    v_0|_{\Sigma} := f_0 = H_3^{(x_0, t_0)}|_{\Sigma}
\end{equation*}
as the auxiliary function.

Now, when $\tau$ is large, the product $v_0 v_1 v_2^2 \approx \tau^{\frac{3}{2}} \e^{-\tau (|x-x_0|^2 + (t-t_0)^2)}$ near the point $(x_0, t_0)$ and by 
the integral identity \eqref{eq: integralidentity: eps general case} and Lemma \ref{lemma:approx identity q} we have
\begin{equation}\label{eq: recoequ pw}
    3! \pi^{\frac{3}{2}} q(x_0,t_0)= \lim_{\tau\to \infty}
\int_{\Sigma} f_0 \partial^{(1,2)}_{\eps_1, \eps_2}\big|_{\eps_1 = \eps_2 = 0} \Lambda(\eps_1 f_1 + \eps_2 f_2 ) \diff S(x).
\end{equation}
as the reconstruction formula for $q(x_0, t_0)$. This is then repeated for all points in the admissible reconstruction area $W$.

\section{Numerical implementation}\label{sec: numerical implem}

\subsection{Discretization of the forward model and evaluating the DN-map}\label{sec: forward model and dn map}

We use finite differences to solve the nonlinear wave equation \eqref{eq: nonlinear wave} numerically in a 2D rectangle ($n=2$). For a more in-depth discussion about finite difference methods for the wave equation, we refer the reader to \cite{mitchell_finite_1980}.

Let $N_{x_1}, N_{x_2}$ and $N_t$ be positive integers, $T>0$, and let $\Omega = [a_1, b_1] \times [a_2, b_2]$ with $a_1<b_1$ and $a_2<b_2$. We divide the spatial domain $\Omega$ and the time domain $[0, T]$ into uniformly spaced grids 
\begin{align*}
    x_1^{(i)} &= a_1 + (i-1) \Delta x_1,\quad  i = 1, \ldots, N_{x_1}, \\
    x_2^{(j)} &= a_2 + (j-1) \Delta x_2,\quad  j = 1, \ldots, N_{x_2}, \\
    t^{(k)} &= (k-1) \Delta t,\quad  k = 1, \ldots, N_{t},
\end{align*}
where $\Delta x_1 = \frac{b_1 - a_1}{N_{x_1} - 1}$, $\Delta x_2 = \frac{b_2 - a_2}{N_{x_2} - 1}$ and $\Delta t = \frac{T}{N_t - 1}$. In what follows, we are going to use fourth order central differences to discretize second order spatial derivatives. Therefore,  we introduce ghost points outside the domain $\Omega$ for $i = 0, N_{x_1} + 1$ and $j = 0, N_{x_2} + 1$ in anticipation of setting the boundary values of the nonlinear wave equation.  For a wave equation in two space dimensions, the Courant-Friedrichs-Lewy (CFL) number is defined by $c = \frac{\Delta t}{\Delta x_1} + \frac{\Delta t}{\Delta x_2}$. In the numerical implementation we choose the numbers $N_{x_1}, N_{x_2}$ and $N_t$ so that the CFL condition $c \leq \frac{1}{\sqrt{2}}$ is satisfied \cite[Chapter 4.12]{mitchell_finite_1980}.

The discrete equation that we use to represent the nonlinear wave equation \eqref{eq: nonlinear wave} is
\begin{align}\label{eq: discrete_nonlinwave}
\begin{split}
    &\frac{u_{i, j, k+1} - 2 u_{i, j, k} + u_{i, j, k-1}}{\Delta t^2} \\
    &= \frac{-u_{i+2, j, k} + 16 u_{i+1, j, k} - 30 u_{i, j, k} + 16 u_{i-1, j, k} - u_{i-2, j, k}}{12 \Delta x_1^2} \\
    &+ \frac{-u_{i, j+2, k} + 16 u_{i, j+1, k} - 30 u_{i, j, k} + 16 u_{i, j-1, k} - u_{i, j-2, k}}{12 \Delta x_2^2} \\
    &- q_{i, j, k} u_{i, j, k}^p.
\end{split}
\end{align}
Here, we have approximated the second order derivatives using second-order central differences for the time derivative and fourth-order central differences for the spatial derivatives. Higher order differences for the space variables are used to enhance the accuracy of the spatial derivatives, which are especially important for evaluating the normal derivatives on the boundary for the DN-map.  Now, a solution to equation \eqref{eq: nonlinear wave} satisfies the discrete equation \eqref{eq: discrete_nonlinwave} up to an error of scale $\Delta t^2 + \Delta x_1^4 + \Delta x_2^4$.

To enforce the initial condition of equation \eqref{eq: nonlinear wave}, we initialize
\begin{align*}
    u_{i, j, 1} = u_{i, j, 2} = 0
\end{align*}
for all $i = 0, \ldots, N_{x_1} + 1$, $j = 0, \ldots, N_{x_2} + 1$, and to enforce the boundary condition of equation \eqref{eq: nonlinear wave}, we set
\begin{align*}
    u_{i, j, k} = f(x_i, x_j, t_k)
\end{align*}
for $i = 0, 1, N_{x_1}, N_{x_1} + 1$, \hspace{0.1cm} $j = 0, 1, N_{x_2}, N_{x_2} + 1$, and all $k = 1, \ldots, N_t$.

The values for $u_{i, j, k}, i = 2, \ldots, N_{x_1} - 1, j = 2, \ldots, N_{x_2} - 1$ are solved iteratively for each time $k = 3, \ldots, N_t$ using the values $u_{i, j, k}, i = 0, \ldots, N_{x_1} + 1$, $j = 0, \ldots, N_{x_2} + 1$ for times $k - 1$ and $k - 2$. The iterative formula is easily attained by solving for $u_{i, j, k+1}$ in \eqref{eq: discrete_nonlinwave}.

We numerically compare the convergence of the finite difference approximation \eqref{eq: discrete_nonlinwave} and the classical approximation that uses second-order central differences for both time and space. Let $U_{2,N}, U_{4,N}$ denote the numerical solutions of the nonlinear wave equation, with the domain $\Omega \times [0,T]$ divided into $N = N_{x_1} \cdot N_{x_2} \cdot N_t$ nodes. The two results are obtained with the second order central difference approximation and the approximation \eqref{eq: discrete_nonlinwave}, respectively. To study the convergence, we progressively increase the number of nodes $N$ and compare the resulting wave equation approximation to a solution $U_{4, 451\cdot451\cdot9001}$, which was computed with a large number of nodes. To measure the convergence, we use the norm
\begin{equation*}
    \Vert U_{4, 451\cdot451\cdot9001} - U_{\cdot, N} \Vert_{\ell^1(X)} := \sum_{(x_1, x_2, t) \in X} |U_{4, 451\cdot451\cdot9001}(x_1, x_2, t) - U_{\cdot, N}(x_1, x_2, t)|
\end{equation*}
where $X$ is a set of points that are included in every discretization of $\Omega \times [0, T]$. In our convergence tests, the size of $X$ is $1000$ nodes.
For each discretization, we choose $N_{x_1} = N_{x_2}$ and then set $N_t$ such that the CFL condition $c \approx 0.3$. 
The results of the convergence tests are depicted in Figure~\ref{fig: convergence of wave eq solutions}.

\begin{figure}[ht]
    \centering
    \includegraphics[width=10cm]{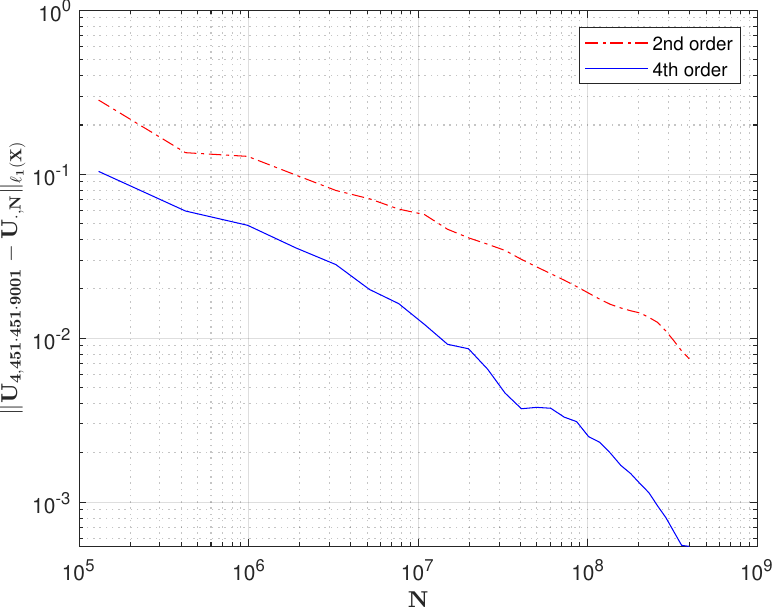}
    \caption{Comparison of the convergence of two finite difference approximation solutions to the nonlinear wave equation as a function of the total number of nodes in the discretization. We see that the finite difference approximation given by \eqref{eq: discrete_nonlinwave} (blue solid line) achieves a given accuracy with fewer discretization nodes compared to the classical second-order solution (red dot-dashed line).
    }
    \label{fig: convergence of wave eq solutions}
\end{figure}

To evaluate the DN-map $\Lambda_q$, we approximate the normal derivatives on the boundary $\Sigma = (\left( \{a_1, b_1\} \times [a_2, b_2] \right) \cup \left( [a_1, b_1] \times \{a_2, b_2\} \right)) \times [0, T]$ by
\begin{equation}\label{eq: normalderivatives_a1b1}
\begin{split}
    \partial_\nu u(a_1, x_2^{(j)}, t^{(k)}) &\approx -\frac{u_{3, j, k} - u_{1, j, k}}{2 \Delta x_1}, \\
    \partial_\nu u(b_1, x_2^{(j)}, t^{(k)}) &\approx \frac{u_{N_{x_1}, j, k} - u_{N_{x_1} - 2, j, k}}{2 \Delta x_1}
    \end{split}
\end{equation}
for all $j = 1, \ldots, N_{x_2}$, $k = 1, \ldots, N_t$ and
\begin{align}\label{eq: normalderivatives_a2b2}
\begin{split}
    \partial_\nu u(x_1^{(i)}, a_2, t^{(k)}) &\approx -\frac{u_{i, 3, k} - u_{i, 1, k}}{2 \Delta x_2}, \\
    \partial_\nu u(x_1^{(i)}, b_2, t^{(k)}) &\approx \frac{u_{i, N_{x_2}, k} - u_{i, N_{x_2} - 2, k}}{2 \Delta x_2}
\end{split}
\end{align}
for all $i = 1, \ldots, N_{x_1}$, $k = 1, \ldots, N_t$.
These normal derivatives will be used later as the synthetic data in the inverse problem.

\subsection{Numerical approximation of the derivatives}\label{sec: numerical differentiation}

Both the reconstruction of the Radon transform \eqref{eq: recoequ Radon} and the pointwise reconstruction method \eqref{eq: recoequ pw} call for numerical evaluation  of the (mixed) derivative 
\begin{equation}\label{eq: third derivative of noisy lambda}
	\partial^{\sigma}_{\vec{\eps}}\big|_{\vec{\eps} = \vec{0}} \widetilde{\Lambda}_q (f),
\end{equation}
of the DN-map 
\begin{equation}\label{eq: noisy DN}
	\widetilde{\Lambda}_q (f) := \Lambda_q (f) + \mathcal{E} (f)
\end{equation}
that is corrupted by measurement error $\mathcal{E}: H^s(\Sigma)\to L^2(\Sigma)$. Later on, we will assume that the error satisfies a bound
\[
\Vert \mathcal{E}(f)\Vert_{L^2(\Sigma)}\leq \delta
\]
for some $\delta>0$ and all sufficiently small boundary values $f\in H^s(\Sigma)$.

The numerical differentiation of data corrupted by noise is a classical ill-posed problem in the sense that a small amount of noise in the data can result in a large error in the approximate derivative. The error can be even more prominent when numerically evaluating higher-order derivatives, such as \eqref{eq: third derivative of noisy lambda}.
Our approach to regularize the numerical differentiation is based on the Fourier transform and spectral filtering.
Our method builds on the ideas presented in \cite{elden_wavelet_2000, qian_fourier_2006}, where they use similar Fourier-based differentiation with a simple truncating filter. Also in \cite{yang_generalized_2014} they use a similar method with a different choise of the filter function.
In this work, we give rigorous justification and optimal stability estimates for a general family of filter functions and compare with examples the effect of two of these filters.
There are a variety of other methods for numerical differentiation using for example Tikhonov regularization \cite{Cullum_numerical_differentiation} or optimization methods \cite{knowles_variational_diff}. For an overview of different methods, we refer to \cite{knowles_methods_2014, smirnova}, and the references therein.

In this section, we consider finding the $p$th derivative $g^{(p)}(x)$ of a generic function $g \in H^s (\R)$, with $s \geq p$. This is relevant to the derivative we need to evaluate for the reconstruction of the Radon transform since in our implementation the derivative \eqref{eq: third derivative of noisy lambda} gets the form $\partial^3_{\eps_1}\big|_{\eps_1 = 0} \widetilde{\Lambda}_q (\eps_1 f_1)$. It can be noted that we are only interested in the derivative at $\eps_1 = 0$.

For $s \in \R$, the Sobolev space $H^s(\R)$ is equipped with the norm
\begin{equation}\label{eq: Sobolev norm}
\Vert g \Vert_{H^s(\R)} := \bigg( \int_{-\infty}^{\infty} (1 + \xi^2)^s | \widehat{g}(\xi) |^2 \diff \xi \bigg)^{\frac{1}{2}}.
\end{equation}

The \emph{regularized spectral differentiation} method for computing $p$th order derivatives is as follows. Let $s>0$ and $p$ be a non-negative integer with $p\leq s$, and suppose $g \in H^s(\R)$ is an exact data function. Let $\widehat{g}$ denote the Fourier transform of $g$, defined by
\begin{equation*}
    \widehat{g} (\xi) := \FT(g)(\xi) = \frac{1}{\sqrt{2 \pi}} \int_{-\infty}^{\infty} g (x) \e^{-\i \xi x} \diff x.
\end{equation*}
With this normalization, 
the Fourier transform extends to an isometry on $L^2(\R)$.
For the $p$th derivative there holds
\begin{equation}\label{eq: Fourier-derivative no noise}
    g^{(p)}(x) = \FT^{-1}((\i \xi)^p \widehat{g})(x) = \frac{1}{\sqrt{2 \pi}} \int_{-\infty}^{\infty} (\i \xi)^p \widehat{g} (\xi) \e^{\i \xi x} \diff \xi.
\end{equation}
Now, suppose that instead of the exact function $g$, we have a measured data function $g_{\delta} \in H^s(\R)$ satisfying $\Vert g - g_{\delta} \Vert_{L^2(\R)} \leq \delta$ 
with the error level $\delta>0.$ 
The goal is to obtain a good approximation of $g^{(p)}(x)$ from $g_{\delta}(x)$. A natural way to stabilize the problem is to eliminate the high frequencies from the solution with a filter function in the Fourier space.

We call a family of measurable functions $m_R:\R\to\C$ \emph{admissible spectral filters of order $p$ on $H^s$}, if
\begin{equation}\label{eq: bound for general spectral filter}
|\xi|^p |m_R(\xi)| \leq C_1 R^p   
\end{equation}
and
\begin{equation}\label{eq: bound for (1 - m) for general spectral filter}
|1-m_R(\xi)| \leq C_2 \min\left\{ 1, \left(\frac{|\xi|}{R}\right)^{s-p} \right\},
\end{equation}
for some constants $C_1, C_2>0$,
and define the regularized differentiation operator $D_R^{(p)}: L^2(\R) \to L^2(\R)$ of order $p$ on $H^s$ as the Fourier multiplier
\begin{equation}\label{eq: regularized diff operator D_R}
\widehat{D^{(p)}_R g}(\xi) := (\i\xi)^pm_R(\xi) \widehat g(\xi).
\end{equation}

Next, we will prove the stability of the regularized differentiation operator $D_R^{(p)}$. For this, we need the following two Lemmas. Lemma \ref{lemma: noise amplification of general regul. diff. oper.} provides a bound for the noise amplification of $D_R^{(p)}$ and Lemma \ref{lemma: approximation error of general regul. diff. oper.} gives a bound for the approximation error of $D_R^{(p)}$.

\begin{lemma}\label{lemma: noise amplification of general regul. diff. oper.}
    Let $s >0$ and $p$ be an integer such that $0 \leq p < s$ and $\delta, E > 0$.
    Let $g \in H^s(\R)$ and $g_\delta \in L^2(\R)$ with $\Vert g - g_\delta \Vert_{L^2(\R)}<\delta$.
    Then
    \begin{equation*}
        \Vert D^{(p)}_R g - D^{(p)}_R g_{\delta} \Vert_{L^2(\R)} \leq C_1 E^{p/s} \delta^{1 - p/s},
    \end{equation*}
    when $R>0$ is chosen as $R = \left( \frac{E}{\delta} \right)^{1/s}$.
\end{lemma}
\begin{proof}
    Using Plancherel's theorem and the bound \eqref{eq: bound for general spectral filter} we get
    \begin{align*}
        \big\Vert D^{(p)}_R g - D^{(p)}_R g_{\delta} \big\Vert_{L^2(\R)} &= \left\Vert \widehat{D^{(p)}_R g} - \widehat{D^{(p)}_R g_{\delta}} \right\Vert_{L^2(\R)}\\
        &= \Vert (\i \xi)^{(p)} m_R(\xi) (\widehat{g} - \widehat{g_{\delta}}) \Vert_{L^2(\R)}\\
        &\leq C_1 R^p \Vert g - g_{\delta} \Vert_{L^2(\R)} < C_1 R^p \delta.
    \end{align*} 
    Setting $R = \left( \frac{E}{\delta} \right)^{1/s}$, we obtain the result.
\end{proof}

\begin{lemma}\label{lemma: approximation error of general regul. diff. oper.}
    Adopt the notation and assumptions of Lemma \ref{lemma: noise amplification of general regul. diff. oper.}. 
    Assume in addition that $\Vert g \Vert_{H^s(\R)}\leq E$.
    Then
    \begin{equation*}
        \Vert g^{(p)} - D^{(p)}_R g \Vert_{L^2(\R)} \leq C_2 E^{p/s} \delta^{1 - p/s},
    \end{equation*}
    when $R>0$ is chosen as $R = \left( \frac{E}{\delta} \right)^{1/s}$.
\end{lemma}
\begin{proof}
    As in the proof of Lemma \ref{lemma: noise amplification of general regul. diff. oper.}, using the Plancherel theorem we estimate
    \begin{align*}
        \Vert g^{(p)} - D^{(p)}_R g \Vert_{L^2(\R)}^2
        &= \Vert (\i \xi)^p (1 - m_R(\xi)) \widehat{g} \Vert_{L^2(\R)}^2\\
        &= \int^{\infty}_{-\infty} \xi^{2p} (1 - m_R(\xi))^2 |\widehat{g}|^2 \diff \xi\\
        &= \int^{\infty}_{-\infty} \frac{\xi^{2p} (1 - m_R(\xi))^2}{(1 + \xi^2)^s} (1 + \xi^2)^s |\widehat{g}|^2 \diff \xi\\
        &\leq \sup_{\xi} \bigg(\frac{\xi^{2p}}{(1 + \xi^2)^s} (1 - m_R(\xi))^2 \bigg) \Vert g \Vert_{H^s(\R)}^2\\
    \end{align*}

    Let us estimate the function $\xi \mapsto \frac{\xi^{2p}}{(1 + \xi^2)^s} (1 - m_R(\xi))^2$. We divide the estimation into two parts. Consider first 
    $|\xi| > R$. 
    Now $\frac{|\xi|}{R} > 1$ and by the bound \eqref{eq: bound for (1 - m) for general spectral filter}
    \begin{equation*}
        \frac{\xi^{2p}}{(1 + \xi^2)^s} (1 - m_R(\xi))^2 \leq C_2 \xi^{2(p - s)} = C_2 \frac{1}{\xi^{2(s - p)}} \leq C_2 \frac{1}{R^{2(s - p)}} = C_2 R^{2(p - s)}.
    \end{equation*}
    Next, suppose 
    $|\xi| \leq R$. 
    Now $\frac{|\xi|}{R} \leq 1$ and by the bound \eqref{eq: bound for (1 - m) for general spectral filter}
    \begin{equation*}
        \frac{\xi^{2p}}{(1 + \xi^2)^s} (1 - m_R(\xi))^2 \leq C_2 \xi^{2(p - s)} \left( \frac{\xi}{R} \right)^{2(s - p)} = C_2 R^{2(p - s)}.
    \end{equation*}
    Thus $\sup_{\xi} \big(\frac{\xi^{2p}}{(1 + \xi^2)^s} (1 - m_R(\xi))^2 \big) \leq C_2 R^{2(p - s)}$.

    The claim follows by setting $R=(\frac{E}{\delta})^{1/s}$ and using the bound $\Vert g \Vert_{H^s(\R)} \leq E$.
\end{proof}

Combining Lemmas~\ref{lemma: noise amplification of general regul. diff. oper.} and \ref{lemma: approximation error of general regul. diff. oper.}, we formulate a stability estimate for regularized spectral differentiation.

\begin{thm}\label{thm: stability estimate for general regul. diff. oper.}
    Let $s>0$ and $p$ be an integer such that $0 \leq p < s$. Then there is a constant $C>0$ depending only on $s$ and $p$ such that the following holds.
    Let $f\in H^s(\R)$ with $\Vert f \Vert_{H^s(\R)}\leq E$, $\delta>0$, and $f_\delta\in L^2(\R)$ with $\Vert f - f_\delta\Vert_{L^2(\R)}<\delta$. Then
    \begin{equation*}
    \Vert f^{(p)} - D_R^{(p)}f_\delta \Vert_{L^2} \leq CE^{p/s} \delta^{1 - p/s},
    \end{equation*}
    when $R>0$ is chosen as $R = \left( \frac{E}{\delta} \right)^{1/s}$.
\end{thm}
\begin{proof}
    Triangle inequality and Lemmas~\ref{lemma: noise amplification of general regul. diff. oper.} and \ref{lemma: approximation error of general regul. diff. oper.}.
\end{proof}

The above bound can be contrasted with the Gagliardo-Nirenberg interpolation inequality~\cite[Theorem 1]{brezis_gagliardo_2018}, which says
\begin{equation*}
\Vert f^{(p)} \Vert_{L^2(\R)} \leq C \Vert f \Vert_{H^s(\R)}^{p/s} \Vert f \Vert_{L^2(\R)}^{1-p/s}.
\end{equation*}
In fact, in the $L^2$-based case the proof of this statement follows directly from H\"older's inequality and Plancherel theorem:
\begin{align*}
\Vert f^{(p)}\Vert_{L^2(\R)} &= \left(\int_\R |\xi|^{2p} |\widehat f(\xi)|^2 \diff\xi\right)^\frac12
\leq
\left(\int_\R
\Big((1+\xi^2)^s |\widehat f (\xi)|^2 \Big)^\frac{p}{s} |\widehat f (\xi)|^{2(1-\frac{p}{s})} \diff \xi\right)^\frac12\\
&\leq
\left(
\int_\R
(1+\xi^2)^s |\widehat f (\xi)|^2\diff \xi
\right)^\frac{p}{s}
\Vert \widehat f \Vert_{L^2(\R)}^{1-\frac{p}{s}}
=\Vert f \Vert_{H^s(\R)}^{p/s} \Vert f \Vert_{L^2(\R)}^{1-p/s}.
\end{align*}

For any linear operator $\mathcal{A}: L^2(\R) \to L^2(\R)$, we also have the following result.

\begin{thm}
    Let $s>0$ and $p$ be an integer such that $0 \leq p < s$. 
    Then there exists a constant $c_0>0$ depending only on $s$ and $p$ such that the following holds.
    For every linear operator $\mathcal A: L^2(\R)\to L^2(\R)$ and every $E,\delta>0$ with $\delta\leq E$ there exists functions $f\in H^s(\R)$ and $f_\delta\in L^2(\R)$ satisfying $\Vert f \Vert_{H^s(\R)}\leq E$ and $\Vert f - f_\delta\Vert_{L^2(\R)}<\delta$ for which
    \begin{equation*}
    \Vert f^{(p)} - \mathcal Af_\delta \Vert_{L^2} \geq c_0 E^{p/s} \delta^{1 - p/s}.
    \end{equation*}
\end{thm}
\begin{proof}
    
    Let $\phi \in C^{\infty}_0(\R)$ be a nonzero test function with $\norm{\phi}{L^2} = 1$. Let us choose $f_{\delta} \equiv 0$ and $f(x) = a \phi(\lambda x)$ for some $a > 0, \lambda \geq 1$ to be fixed later. Now
    \begin{align*}
        \norm{f - f_{\delta}}{L^2}^2 
        &= \norm{f}{L^2}^2 
        = \int_{-\infty}^{\infty} a^2 \Big| \phi(\lambda x) \Big|^2 \diff x
        = a^2 \frac{1}{\lambda} \norm{\phi}{L^2}^2 
        = \frac{a^2}{\lambda}
    \end{align*}
    and
    \begin{align*}
        \norm{f^{(p)}}{L^2}^2
        &= \int_{-\infty}^{\infty} a^2 \Big| D^p_x \big(\phi(\lambda x)\big) \Big|^2 \diff x
        = a^2 \lambda^{2p -1} \norm{\phi^{(p)}}{L^2}^2
    \end{align*}
    and
    \begin{align*}
        \norm{f}{H^s}^2 
        &= \int^{\infty}_{-\infty} \Big( 1 + \xi^2 \Big)^s \left| \widehat{a \phi(\lambda x)}(\xi) \right|^2 \diff \xi
        = \frac{a^2}{\lambda} \int^{\infty}_{-\infty} \Big( 1 + (\lambda \eta)^2 \Big)^s \left| \widehat{\phi}(\eta) \right|^2 \diff \eta\\
        &\leq a^2 \lambda^{2s -1} \int^{\infty}_{-\infty} \Big( 1 + \eta^2 \Big)^s \left| \widehat{\phi}(\eta) \right|^2 \diff \eta
        = a^2 \lambda^{2s -1} \norm{\phi}{H^s}^2.
    \end{align*}

    Let us set $\lambda = \left( \frac{E}{\delta} \right)^{1/s} \geq 1$ and select the constant $a>0$ so that $\norm{f - f_{\delta}}{L^2} < \delta$ and $\norm{f}{H^s} \leq E$. These conditions are achieved by choosing $a$ such that $\frac{a}{\sqrt{\lambda}} \leq \frac{\delta}{2}$ and $a \lambda^{s - 1/2} \norm{\phi}{H^s} \leq E$, so
    \begin{equation*}
        a \leq \frac{\delta \sqrt{\lambda}}{2} = \frac{1}{2} \delta^{1 - 1/2s} E^{1/2s} \text{ and } a \leq \frac{E \lambda^{1/2 - s}}{\norm{\phi}{H^s}} = \frac{1}{\norm{\phi}{H^s}} \delta^{1 - 1/2s} E^{1/2s}.
    \end{equation*}
    We set
    \begin{equation*}
        a = \delta^{1 - 1/2s} E^{1/2s} \min\left\{\frac{1}{2}, \frac{1}{\norm{\phi}{H^s}}\right\}.
    \end{equation*}
    Finally
    \begin{align*}
        \norm{f^{(p)} - \mathcal{A} f_{\delta}}{L^2} 
        &= \norm{f^{(p)}}{L^2} 
        = a \lambda^{p - \frac{1}{2}} \norm{\phi^{(p)}}{L^2}\\
        &= \delta^{1 - 1/2s} E^{1/2s} \min\left\{\frac{1}{2}, \frac{1}{\norm{\phi}{H^s}}\right\} \left( \frac{E}{\delta} \right)^{p/s - 1/2s} \norm{\phi^{(p)}}{L^2}\\
        &\geq c_0 E^{p/s} \delta^{1 - p/s}
    \end{align*}
    for $c_0 = \frac{1}{2} \min\left\{\frac{1}{2}, \frac{1}{\norm{\phi}{H^s}}\right\} \norm{\phi^{(p)}}{L^2}$.

    Thus, we have found a constant $c_0$ independent of the linear operator $\mathcal{A}$ and the bounds $E, \delta$, such that the claim is satisfied.
\end{proof}
\begin{remark}
    If $E < \delta$, the result becomes trivial, because now for any bounded linear operator $\Vert f^{(p)}-\mathcal{A}(f_\delta)\Vert_{L^2(\R)} \leq C\delta$ and there is nothing to prove.
\end{remark}

Therefore, we obtain minimax optimality:
\begin{equation*}
\inf_{\mathcal A \in \mathcal L(L^2,L^2)}
\sup_{\substack{
    \|f\|_{H^s} \le E \\
    \|f - f_\delta\|_{L^2} \le \delta
}}
\| f^{(p)} - \mathcal A f_\delta \|_{L^2}
 \asymp E^{p/s} \delta^{1-p/s}.
\end{equation*}
No linear operator can beat this rate, while $D_R^{(p)}$ associated with an admissible spectral filter $m_R$ satisfies it. Thus, within the class of bounded translation-invariant linear operators on $L^2(\R)$ (i.e. Fourier multipliers with bounded symbols \cite[Theorem~3.16 of Chapter~1]{stein_introToFourier_1971}), admissible spectral filters provide a natural and optimal realization of regularized 
differentiation of order $p$.

\begin{example}
We finish with two examples; the Truncating filter $m_R(\xi) = \mathbf{1}_{[-R,R]}(\xi)$ and the Gaussian filter $m_R(\xi)=\e^{-(\xi/R)^2}$. It is easy to verify that these filters satisfy the conditions \eqref{eq: bound for general spectral filter} and \eqref{eq: bound for (1 - m) for general spectral filter} for the admissible spectral filters. For the Gaussian filter to satisfy condition \eqref{eq: bound for (1 - m) for general spectral filter}, we need to set a lower bound $p \geq s - 2$ for the order of the derivative $p$. Thus, we require $s - 2 \leq p < s$.    
\end{example}

By Theorem~\ref{thm: stability estimate for general regul. diff. oper.}, the stability estimates for regularized derivatives are the same whether using truncation or a Gaussian filter for the filter function $m_R(\xi)$.
Truncation, however, can cause Gibbs phenomenon and overshoot near discontinuities~\cite{Jerri_Gibbs_book}, while a Gaussian filter avoids oscillations but may oversmooth features, as it corresponds to convolution with a Gaussian in the spatial domain.

As with most regularization methods, choosing the regularization parameter is not straightforward. In our case Theorem~\ref{thm: stability estimate for general regul. diff. oper.} gives a theoretically justified choice for the regularization parameter
$R > 0$
depending on the error level $\delta>0$, but in practice $\delta$ is usually unknown, and the choice of
$R$
is made experimentally.
For theory on different methods for choosing the regularization parameter, we refer to \cite{hansen2010discreteinversebook, mueller2012linearandnonlinearinversebook}. In the examples below, and in our inverse problem implementation, we use experimentally chosen regularization parameters.

In the numerical implementation of the inverse problem, we approximate the third derivative in \eqref{eq: recoequ Radon} using the regularized third derivative \eqref{eq: regularized diff operator D_R} with a Gaussian filter. Fourier transforms are computed via the fast Fourier transform (FFT), which assumes periodic data. Since this is not realistic for our measurements, we enforce periodicity by multiplying the data $g$ with a smooth cutoff $\Psi\in C_0^\infty(a,b)$ satisfying $\Psi(x)=1$ near the point of interest $x_0\in(a,b)$. The product $g_{\text{periodic}} = \Psi g$ agrees with $g$ in a neighborhood of $x_0$ and extends periodically from $[a,b]$ to $\R$. If derivatives are required over the entire interval $[a,b]$, periodic extension can instead be obtained with smoothing splines, as in the appendix of~\cite{elden_wavelet_2000}.

\begin{example}

    We consider  Example~5.2 of \cite{qian_fourier_2006}, that is, let 
    \[
    g(x)
    =
    \begin{cases}
        0, &  0\leq x < \frac14,\\
        2x^2 - x + \frac18, & \frac14 < x \leq \frac12,\\
        3x-2x^2-\frac78, & \frac12<x\leq\frac34,\\
        \frac14, & \frac34<x\leq 1.
    \end{cases}
    \]
    It can be verified that $g$ belongs to $H^{2+\epsilon}([0,1])$ for $\epsilon<1/2$, and no better. Indeed, $g'$ is a hat-function supported on $[\frac14,\frac34]$ and $g'' =4(\mathbf{1}_{[\frac14,\frac12]}-\mathbf{1}_{[\frac12,\frac34]})$, having jump-discontinuities at $x=\frac14,\frac12,\frac34$, see Figure~\ref{fig:regu_diff_comparison}. Since $g$ is non-periodic, we extend $g$ from $[0,1]$ to $[0,2]$ and then periodically over $\R$ using the method described in the appendix of~\cite{elden_wavelet_2000}.
    
    The numerical examples are done in \textsc{Matlab}, where the discrete version of $g$ is
    \[
    g_\delta = g+\sigma\, \mathtt{randn}(\mathtt{size}(g))
    \]
    with $g=(g(x_1),\ldots,g(x_n))$, $x_j=(j-1)/(N-1)$, $j=1,2,\ldots,N$ ($N=4097$). The command \texttt{randn}(\texttt{size}(g)) generates an array of the same size as $g$ of normally distributed random numbers with $0$ mean and standard deviation~$1$. We compare the truncation method of \cite{qian_fourier_2006} with the Gaussian filtered approach in Figure~\ref{fig:regu_diff_comparison}. 
    \begin{figure}[ht]
        \centering
        \includegraphics[width=0.45\linewidth]{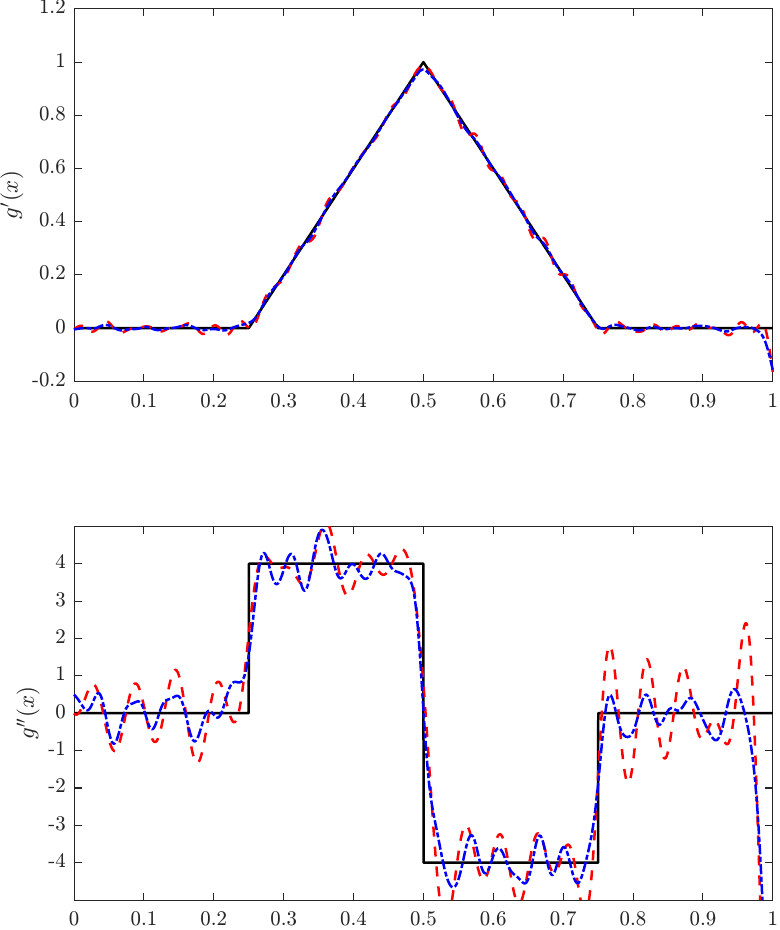}
        \hfill
        \includegraphics[width=0.45\linewidth]{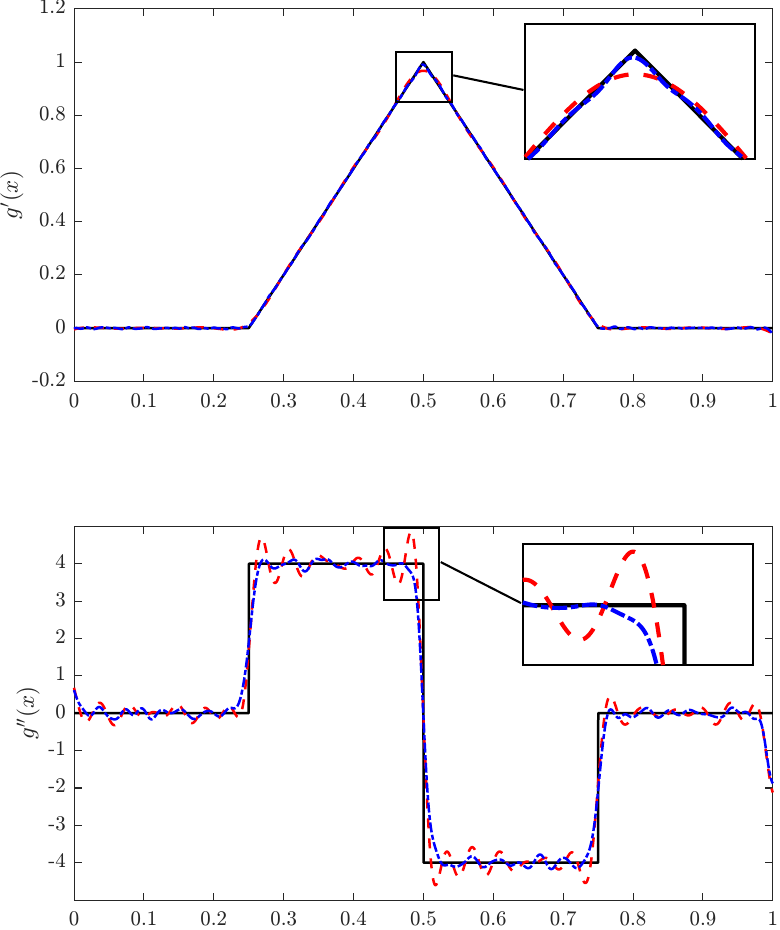}
        \caption{ 
        Comparison of regularized differentiation via high-frequency truncation (red dashed) and Gaussian filtering (blue dot-dashed). The exact derivatives are shown in solid black.
        \textbf{Left:} Noise level $\sigma=10^{-3}$. Truncation cutoffs:
        $R=64$
        (1st derivative) and $40$ (2nd derivative). Gaussian parameters: 
        $R = \sqrt{\frac{10^5}{5}}$
        (1st derivative) and
        $\sqrt{\frac{10^5}{8}}$
        (2nd derivative). RMS-errors: $e_\mathrm{trunc}=0.016$, $e_\mathrm{Gauss}=0.013$ (1st derivative), and $1.54$ vs.\ $1.16$ (2nd derivative).  
        \textbf{Right:} Noise level $\sigma=10^{-4}$. Truncation cutoffs: 
        $R = 200$ 
        (1st derivative) and $87$ (2nd derivative). Gaussian parameter: 
        $R = \sqrt{\frac{10^5}{5}}$
        (both derivatives). RMS-errors: $0.0089$ vs.\ $0.0035$ (1st derivative) and $0.61$ vs.\ $0.49$ (2nd derivative). For the 1st derivative the methods are visually indistinguishable, while for the 2nd derivative the Gaussian filter better suppresses overshoots at jump discontinuities, smoothing the jump in the process.
        }
        \label{fig:regu_diff_comparison}
    \end{figure}

\end{example}

\subsection{Numerical implementation of the inverse problem}\label{sec: numerical implem of the ip}

We consider the nonlinear wave equation \eqref{eq: nonlinear wave eps: general case} with $p = 3$.  For real-valued waves, the cubic nonlinearity coincides with the Kerr-type nonlinearity $|u|^2u=u^3$.

Let us define $g:\R \to \R$ by
\begin{equation}\label{eq: g for radon recoegu}
        g(\eps_1) := \frac{1}{3! \pi} \int_{\Sigma} f_0 \Lambda(\eps_1 f_1) \diff S(x).        
\end{equation}
From \eqref{eq: recoequ Radon} we get the reconstruction formula for $\radonR(q)$ for the parameter tuple $(t_0, \theta_0, \eta_0)$ by
\begin{equation}\label{eq: final recoequ Radon}
    \radonR(q)(t_0, \theta_0, \eta_0) \approx
    \partial^3_{\eps_1}g(0), \quad \tau\gg 1.
\end{equation}
In practice, the measurement of $\Lambda_q$ is corrupted by noise and we only measure $\widetilde \Lambda_q$ as in \eqref{eq: noisy DN}.
Hence, the derivative $\partial^3_{\eps_1}  g(0)$ is approximated using regularized spectral differentiation, following Section~\ref{sec: numerical differentiation}, by measuring $\widetilde g(\eps_1)$ (that is, \eqref{eq: g for radon recoegu} with $\Lambda_q$ replaced by $\widetilde\Lambda_q$). The reconstruction formula can then be written compactly as
\begin{equation}\label{eq: numerical reco}
\radonR(q)^{\text{rec}}(t_0, \theta_0, \eta_0) = D_R^{(3)}\widetilde g(0),
\end{equation}
where for regularized differentiation \eqref{eq: regularized diff operator D_R} we use Gaussian filtering with a regularization parameter $R>0$.

Let $\eps>0$. In the numerical implementation, we divide the interval $(0, \eps]$ into $N_{\eps_1} \in \N$ equally spaced points $\eps_1^{(j)}=j\eps/N_{\eps_1}$, $j=1,\ldots,N_{\eps_1}$. Given the parameters $t_0 \in [0, T]$, $\eta_0\in\R$ and $\theta_0\in\S^1$, we fix the boundary values $\eps_1^{(j)} f_1$ for all $\eps_1^{(j)}$ as in Section \ref{sec: choosing boundary for radon} and evaluate the DN-map $\Lambda_q$ for each boundary value as in Section \ref{sec: forward model and dn map}. In addition, we fix the boundary value $f_0$ for the auxiliary function as in Section~\ref{sec: choosing boundary for radon}. Finally, the required integral in \eqref{eq: g for radon recoegu} over $\Sigma$ is approximated using uniform-grid Riemann sums.

We require the derivative of $\widetilde g$ at the left endpoint $\eps_1 = 0$ of the interval $(0,\eps]$, which is challenging from a numerical standpoint. Therefore, before evaluating the regularized derivative $D_R^{(3)}\widetilde g(0)$, $\widetilde{g}$ is extended to an odd function over the interval $[-\eps, \eps]$ using the fact that $\Lambda_q(\eps_1f)$ is odd in terms of $\eps_1$, since $p = 3$ is odd. Thus, $ g(-\eps_1) = - g(\eps_1)$ for all $\eps_1\in[-\eps,\eps]$. Finally, $\widetilde g(\eps_1)$ is extended to a periodic function over $\R$, as discussed in Section \ref{sec: numerical differentiation}, and the regularized differentiation in \eqref{eq: numerical reco} is done via FFT.

Let us now consider the reconstruction of $\radonR(q)$ for a set of parameters $W_0 \in W_{\text{Radon}}$. Here $W_0 = \mathbf{\Theta} \times P\subset \S^1\times\R$ is the domain of the Radon transform, where $\eta \in P$ is a real number describing a distance from the origin assumed to be within $\Omega$. The discrete reconstruction $\radonR(q)^{\text{rec}}$ is a matrix (sinogram), whose columns correspond to the angles $\Theta$ and the rows correspond to the distances $P$. If $q$ is time-independent, for a given $\theta_0 \in \Theta$, it suffices to make a single measurement of the DN-map corresponding to the parameter tuple $(\frac{T}{2}, \theta_0, 0)$. By varying the auxiliary function $v_0$, we are able to reconstruct the entire column of the sinogram corresponding to the angle $\theta_0$.
More specifically, for each $\eta_0 \in P$, we set the auxiliary function $v_0$ corresponding to $(\frac{T}{2} + \eta_0, \theta_0, \eta_0)$ as in Section \ref{sec: choosing boundary for radon}. With these choices, we can reconstruct $\radonR(q)^{\text{rec}}(\frac{T}{2} + \eta_0, \theta_0, \eta_0) = \radonR(q)^{\text{rec}}(\theta_0, \eta_0)$, where the time parameter can be omitted  since the potential $q$ is assumed to be time-independent. This is then repeated for all $\theta_0 \in \Theta$. The approach is illustrated in Figure~\ref{fig: rec Radon waves} (right). The approach significantly lowers the computational cost of the reconstructions compared to the time-dependent case where a measurement of the DN-map must be made for each parameter tuple in $W_0$.

\begin{figure}[ht]
    \centering
    \includegraphics[width=0.25\linewidth]{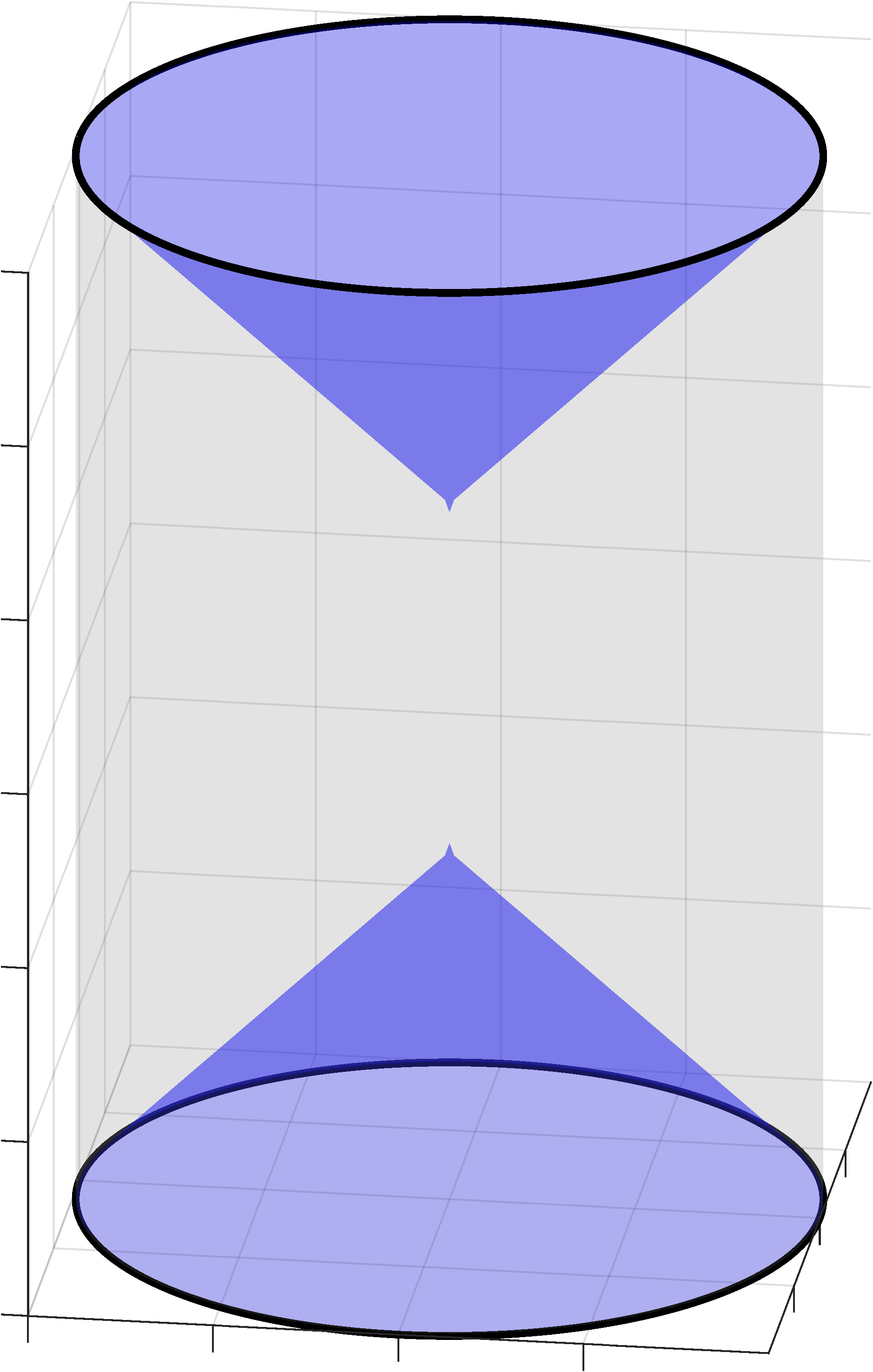}\hspace{2cm}
    \includegraphics[width=0.25\linewidth]{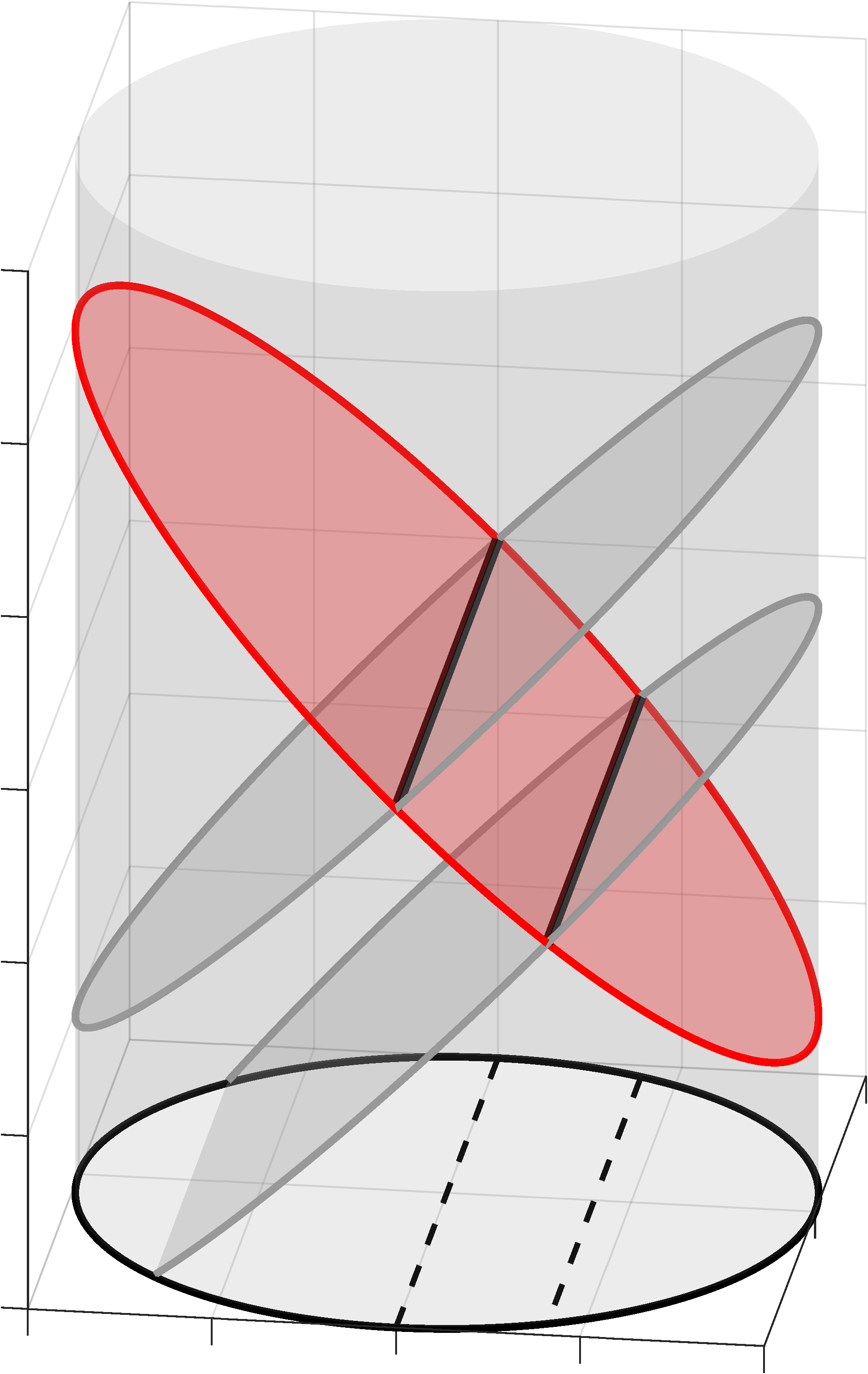}
    \caption{\textbf{Left}: The causal set where waves can be sent and measured lies between the two blue cones within the cylinder $\Omega\times[0,T]$ (gray). This is the optimal set, which can be reached from the lateral boundary $\Sigma$ by light rays and from where light can emanate to the boundary. \textbf{Right}: Illustration of how the Radon transform of $q$ is obtained when $q$ is time-independent. The space–time domain $\Omega\times[0,T]$ is shown as the gray cylinder. A (linear) wave $v_1$ (red plane, bottom-right to top-left) propagates through the domain. Auxiliary waves $v_0$ (gray planes, top-right to bottom-left) intersect $v_1$ along line segments with constant time inside $\Omega$ (solid black lines). These intersection lines are the ones over which $\radonR(q)$ can be computed. Their projections to the initial time plane $t=0$ are indicated by black dashed lines. If $q$ is time-independent, then by varying $v_0$ while fixing a single wave $v_1$ one can recover $\radonR(q)(\theta_0,\eta_0)$  for all $\eta_0$ at a fixed angle $\theta_0\in\S^1$.
    }
    \label{fig: rec Radon waves}
\end{figure}

Finally, once we have the reconstructed Radon transform $\radonR(q)^{\text{rec}}$, filtered backprojection (or another CT reconstruction algorithm) \cite{natterer_mathematicsOfCT_1986} can be used to reconstruct $q^{\text{rec}}$ from $\radonR(q)^{\text{rec}}$.

For a discussion about the computational complexity of the numerical reconstruction scheme, we look at the different steps of the method separately. The complexity of the boundary integration via Riemann sums is $\mathcal O((N_{x_1} + N_{x_2}) N_t)$, while that of the regularized differentiation step using FFT is $\mathcal O(N_\eps\log(N_\eps))$. To achieve the full sinogram of the potential $q$ in one time-step, all of the operations have to be done $N_{\theta} \cdot N_{\eta}$ times, where $N_{\theta}, N_{\eta}$ are the number of columns and rows in the sinogram. Finally, the complexity of the filtered back projection is $\mathcal{O}(N_\theta N_\eta^2)$, see~\cite{roerdink1998data}.

\begin{remark}
    We also implemented the pointwise reconstruction formula \eqref{eq: recoequ pw} to compare it with the Radon approach. The implementation is similar, with the appropriate modifications, to the one described above.
\end{remark}

\section{Numerical examples}\label{sec: numerical examples}

Let $\chi_A$ denote the characteristic function of a set $A \subset \R^2$ and let $\varphi_\text{ellipse} \in C^{\infty}_0(\R^2)$ be a smooth bump function supported inside an ellipse $\{(x,y)\in\R^2 : (x/a)^2 + (y/b)^2<1, a, b > 0\}$ ,  given by
$$
\varphi_\text{ellipse}(x,y)=
\begin{cases}
\exp\left( \dfrac{1}{(x/a)^2 + (y/b)^2-1} \right)
, &(x/a)^2 + (y/b)^2<1 \\
0, &\text{otherwise}.
\end{cases}
$$
Here $a$ and $b$ are the semi-axes of the ellipse. We also apply rotations to the ellipses. Moreover, $\varphi_\text{disc}$ is $\varphi_\text{ellipse}$ with $a = b$.

For the potential $q$, we use the following examples
\begin{enumerate}
    \item\label{eq: example bump} $q(x_1, x_2) = \varphi_{\text{ellipse}_1}(x_1, x_2)$,
    \item\label{eq: example lungs} $q(x_1, x_2) = \varphi_{\text{disc}_1}(x_1, x_2) - \varphi_{\text{ellipse}_2}(x_1, x_2) - \varphi_{\text{ellipse}_3}(x_1, x_2)$,
    \item\label{eq: example disc} $q(x_1, x_2) = \chi_{\text{disc}_2}(x_1, x_2)$,
    \item\label{eq: example L} $q(x_1, x_2) = \chi_L(x_1, x_2) + \varphi_{\text{disc}_2}(x_1, x_2)$,
    \item\label{eq: example sin} $q(x_1, x_2) = \sin(4 \pi k x_1) \cdot \sin(4 \pi k x_2) \cdot \chi_{[-\frac14, \frac14]^2}$, for $k = 2, 3, 4, 5$,
    \item\label{eq: example timedep} $q(x_1, x_2, t) = t \cdot \varphi_{\text{disc}_3}(x_1, x_2)$, where the centre of $\text{disc}_3$ is $(0.3, -0.3) + t(-0.2, 0.2)$.
\end{enumerate}
Ellipses and discs as well as the set $L$ appearing in these expressions are shown below.

For the numerical solution of the forward problem we use the parameters $N_{x_1} = N_{x_2} = 80$ and $N_t = 1500$ in the spatial domain $\Omega = [-0.5, 0.5]^2$ and time domain $[0, T] = [0, 3]$. It follows that the CFL number is $c \approx 0.316$.

For the inverse problem of reconstructing the Radon transform of the unknown potential $q$, we take the reconstruction parameter set $\mathbf{\Theta} \times P$, where 
\[
\mathbf{\Theta} =\{\theta\in\S^1 : \theta=(\cos(t),\sin(t)),\, t\in \Theta'\}
\]
is the set of unit vectors in the directions $\Theta' = \{0^{\circ}, 1^{\circ}, \cdots, 179^{\circ}\}$, and $P$ corresponds to the interval $[-0.4, 0.4]$ divided into $63$ equally spaced distances from the origin. We use experimentally chosen parameters $\eps = 1.5, N_{\eps} = 16$ and $\tau = 700$. The regularization parameter $R = 10$
is chosen for spectral differentiation. To evaluate $q^{\text{rec}}$ from $\radonR(q)^{\text{rec}}$, we use \textsc{Matlab}'s built-in \texttt{iradon}-function implementing the standard filtered back-projection algorithm with the Ram-Lak filter.

For the sake of comparison, we also implemented an (unregularized) pointwise reconstruction algorithm. In this case, we consider a $44 \times 44$ equally spaced reconstruction grid of the domain $[-0.2828, 0.2828]^2$. As the reconstruction parameters we use $\eps = 0.1$ and $\tau = 700$, and as the unit vector $\theta$ we select $\theta = \frac{(1,1)}{\sqrt{2}}$. Numerical integration over $\Sigma$ is carried out with uniform-grid Riemann sums (composite rectangular rule). The parameters here are chosen so that we obtain comparable reconstructions to the Radon approach.

In all examples, the noisy measurements are simulated by adding Gaussian noise with standard deviation
\begin{equation*}
    \sigma_0 = \sigma \cdot \mathtt{mean}(| \Lambda_q(\eps_1 f_1)  |)
\end{equation*}
to the DN-map. Here, $\sigma$ is the noise level and $\mathtt{mean}(| \Lambda_q(\eps_1 f_1)  |)$ is the average magnitude of the exact boundary measurement. We use the noise level $\sigma=2\,\%$ in all our examples. This results in a signal-to-noise ratio of approximately $44\,\mathrm{dB}$.

Figures~\ref{img: recos bump}-\ref{img: recos L} show the reconstructions of Examples~\ref{eq: example bump}-\ref{eq: example L}. We compared the unregularized (naive) finite-differences approach both pointwise and via the Radon transform with the regularized Radon approach. In all cases, the regularized Radon reconstruction provides visually best reconstructions, locating the position, shape, and size of the potential function $q$ accurately. Moreover, in Figure~\ref{img: recos sinograms from lungs} we display the reconstructed sinogram for Example~\ref{eq: example lungs}.
In Figures~\ref{img: recos sin}-\ref{img: recos L limited angles} we further examine the capabilities and limitations of the regularized Radon reconstructions under varying structural complexity (checker board pattern in Example~\ref{eq: example sin}) and data availability (limited incident angles for Example~\ref{eq: example L}). The limited angle experiments displayed in Figure~\ref{img: recos L limited angles} are included to demonstrate the directional visibility phanomena discussed in Remark~\ref{rem: limited tomo}. In the reconstruction method, complete Radon data corresponds to the situation, where waves can be sent into the domain from all incident directions ($\theta \in \mathbb{S}^1$), allowing reconstruction of the full sinogram. In contrast, limited-angle data arises when only a restricted range of incident directions is available. The experiments in Figure~\ref{img: recos sin} showcase how the reconstruction scheme performs on different-frequency potentials when the reconstruction parameters are kept the same. In favorable cases, such as low frequency potential function $q$, the reconstruction is robust, but in the more challenging situations we observe blurring and missing features. Finally, Figure \ref{img: timedependant} shows that the proposed method works equally well for a time-dependent potential.

\begin{figure}[H]
    \centering
    \includegraphics[width=0.32\linewidth]{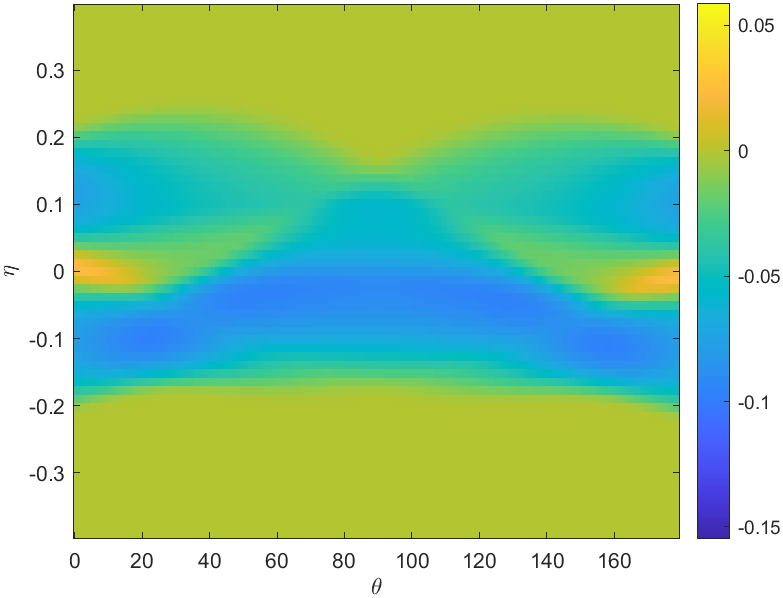}
    \includegraphics[width=0.32\linewidth]{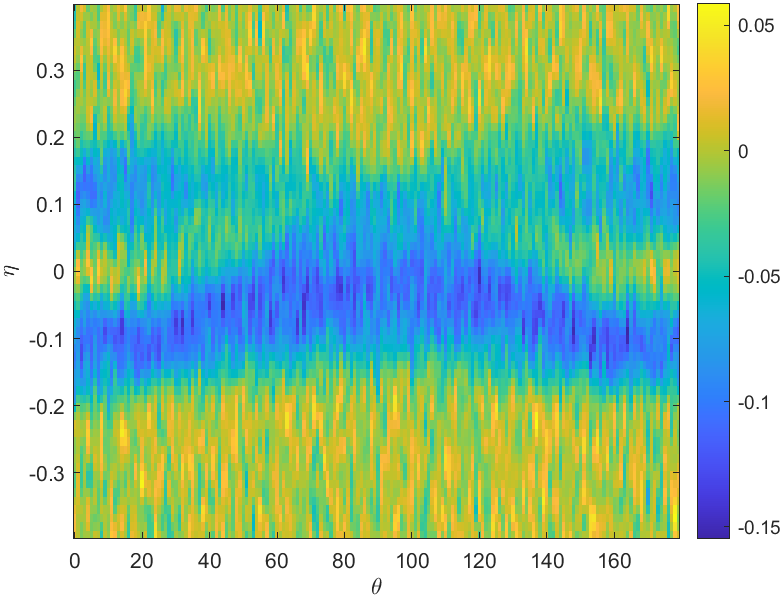}
    \includegraphics[width=0.32\linewidth]{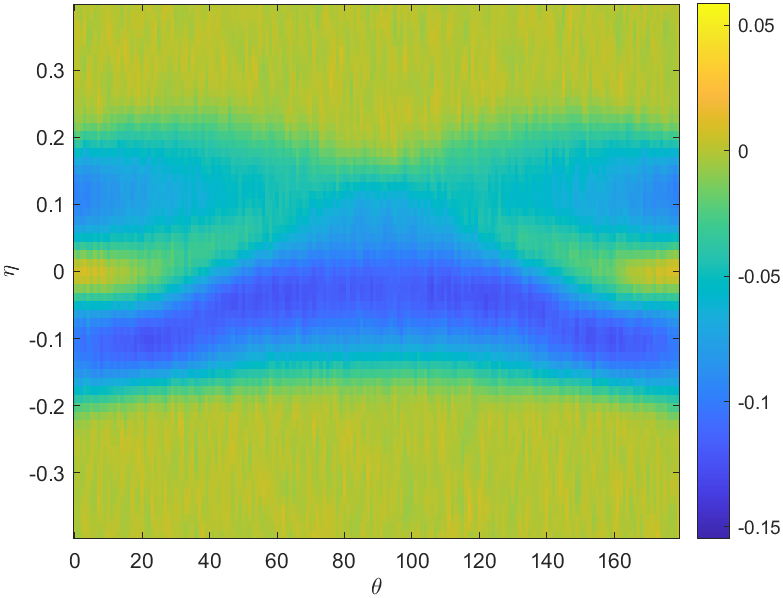}
    \caption{Sinograms of Example~\ref{eq: example lungs}. \textbf{Left:} True $\radonR(q)$ computed with \textsc{Matlab}'s \texttt{radon}. \textbf{Middle:} reconstructed $\radonR(q)^{\text{rec}}$ by finite differences. \textbf{Right:} reconstructed $\radonR(q)^{\text{rec}}$ via spectral differentiation. The finite difference approximation of $D^{(3)}\widetilde g(0)$ amplifies measurement error. A better approximation to the true sinogram is obtained via regularized differentiation as $D_R^{(3)}\widetilde g(0)$. Here 
    $R = 10$.
    }
    \label{img: recos sinograms from lungs}
\end{figure}

\begin{figure}[H]
    \centering
    \includegraphics[width=0.24\linewidth]{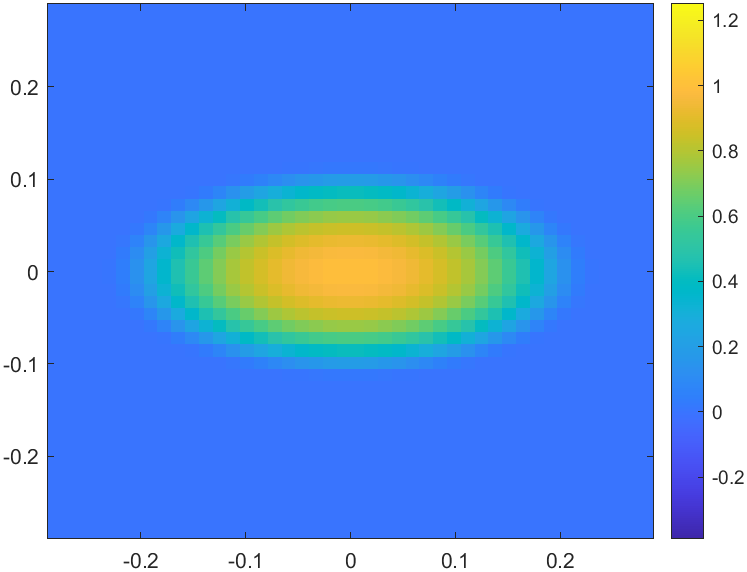}
    \includegraphics[width=0.24\linewidth]{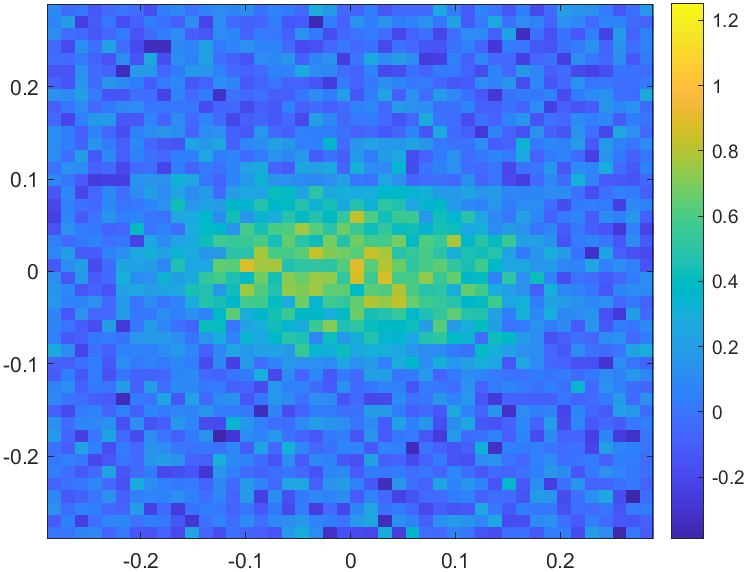}
    \includegraphics[width=0.24\linewidth]{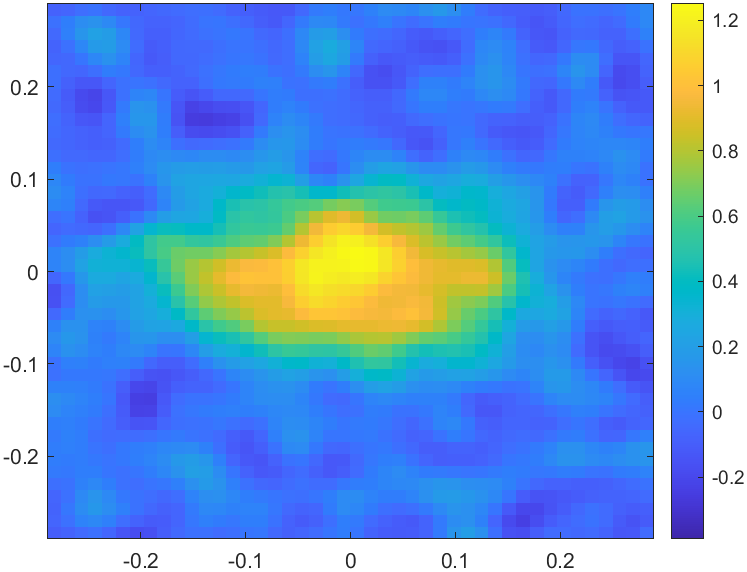}
    \includegraphics[width=0.24\linewidth]{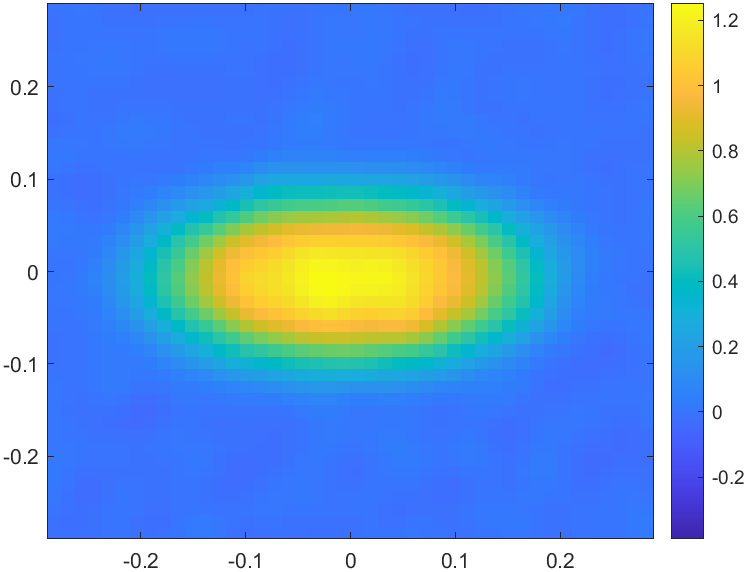}
    \caption{Example \ref{eq: example bump}. \textbf{1st:} True unknown potential. \textbf{2nd:} Pointwise reconstruction using finite differences. \textbf{3rd:} Radon reconstruction using finite differences. \textbf{4th:} Radon reconstruction using regularized spectral differentiation. In the unregularized cases, in the pointwise approach noise dominates the reconstruction compared to the back-projection, where the location and rough shape of the potential can be detected. By regularizing the differentiation step of the reconstruction, the location, shape, and size of the potential function are captured accurately.}
    
    \label{img: recos bump}
\end{figure}

\begin{figure}[H]
    \centering
    \includegraphics[width=0.24\linewidth]{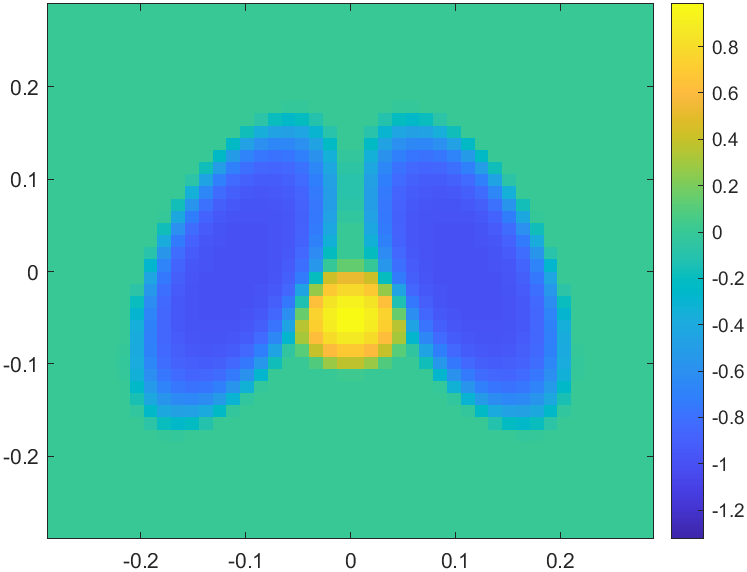}
    \includegraphics[width=0.24\linewidth]{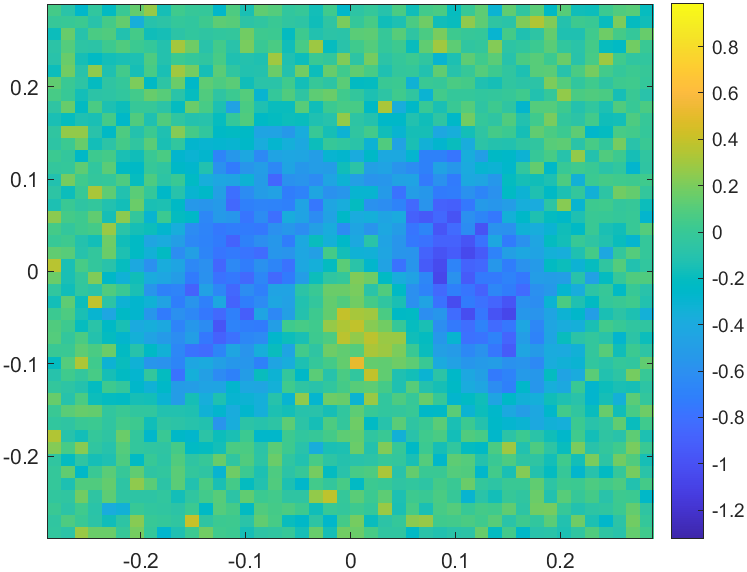}
    \includegraphics[width=0.24\linewidth]{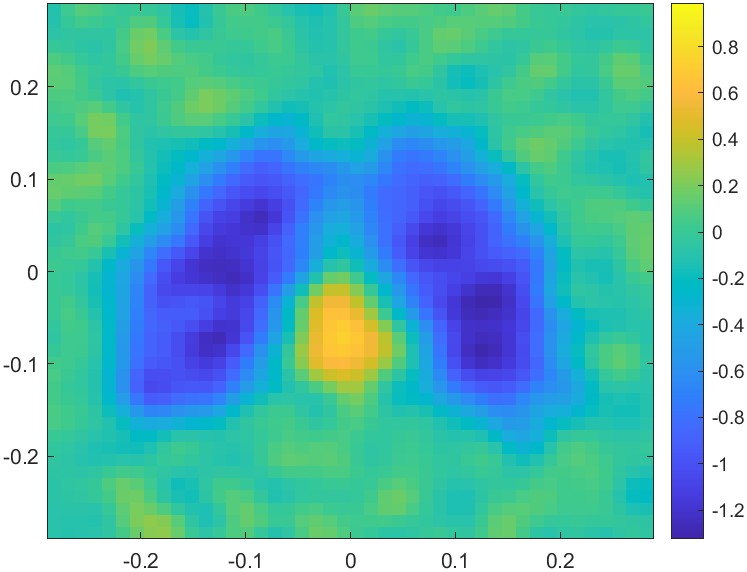}
    \includegraphics[width=0.24\linewidth]{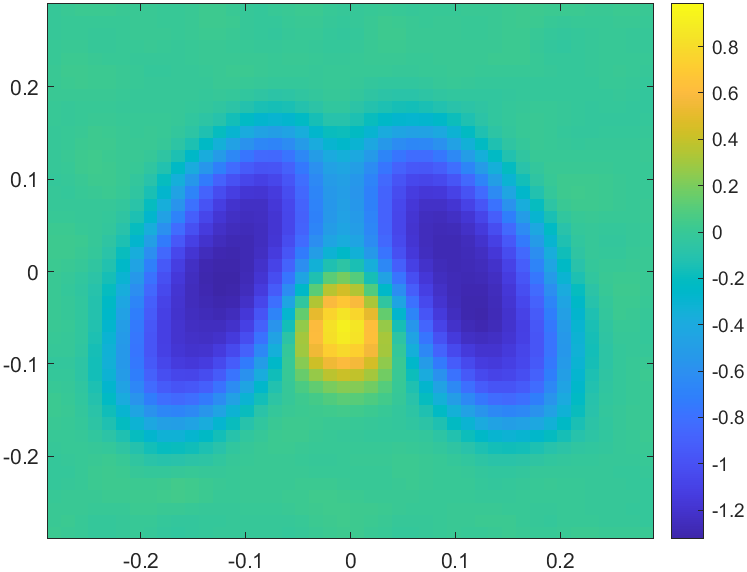}
    \caption{Example \ref{eq: example lungs}, smooth bump functions supported on a disc and two ellipses (tilted by $\pm 22.5^\circ$ from the vertical axis). \textbf{1st:} True unknown potential. \textbf{2nd:} Pointwise reconstruction using finite differences. \textbf{3rd:} Radon reconstruction using finite differences. \textbf{4th:} Radon reconstruction using regularized spectral differentiation. Location, shape, and size of the potential functions are captured accurately, but finer overlapping features appear blurred.}
    
    \label{img: recos lungs}
\end{figure}

\begin{figure}[H]
    \centering
    \includegraphics[width=0.24\linewidth]{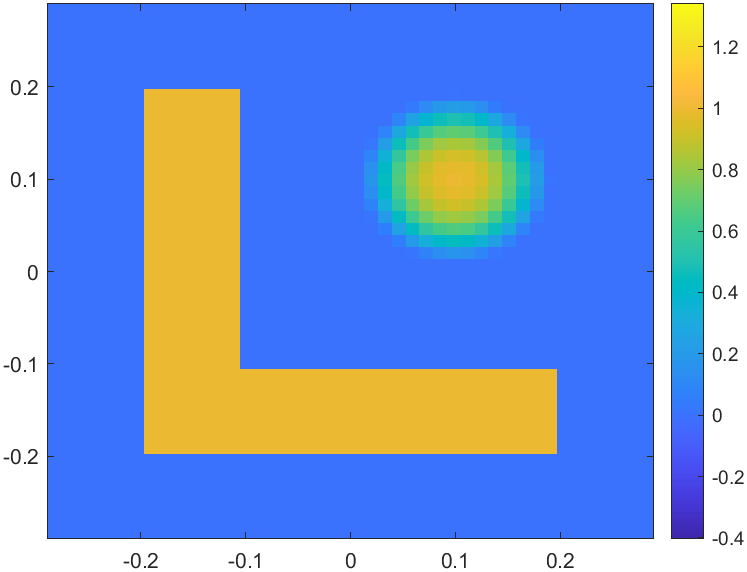}
    \includegraphics[width=0.24\linewidth]{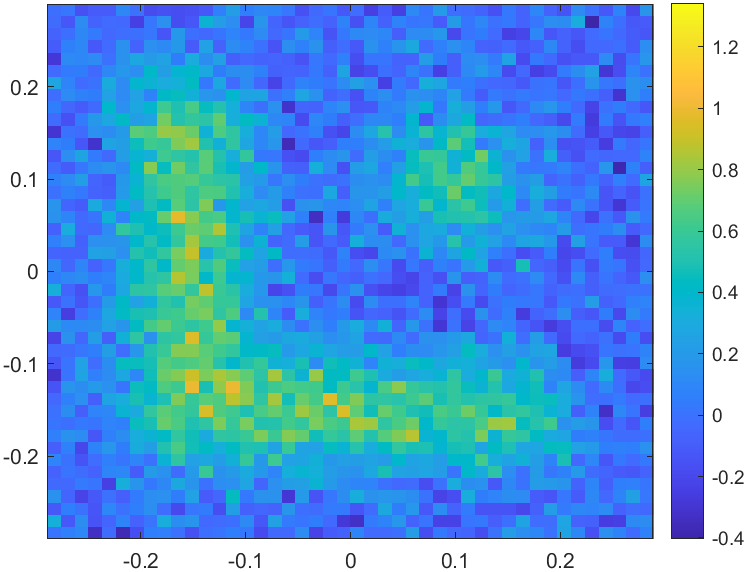}
    \includegraphics[width=0.24\linewidth]{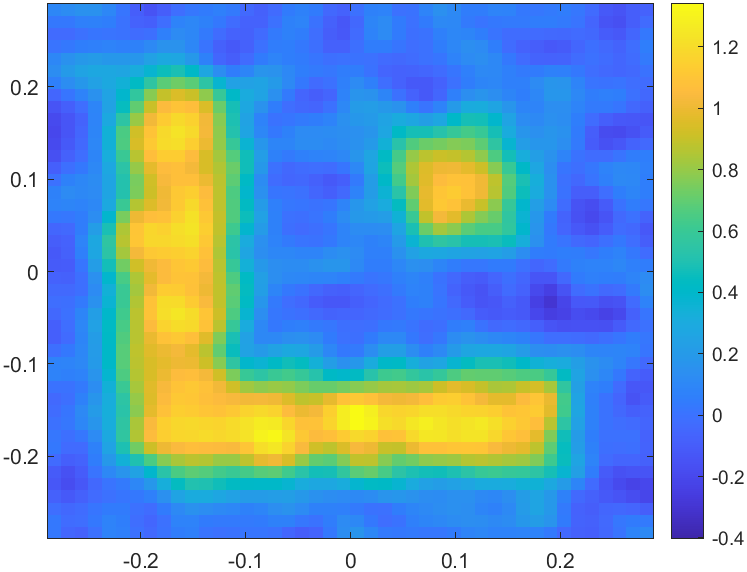}
    \includegraphics[width=0.24\linewidth]{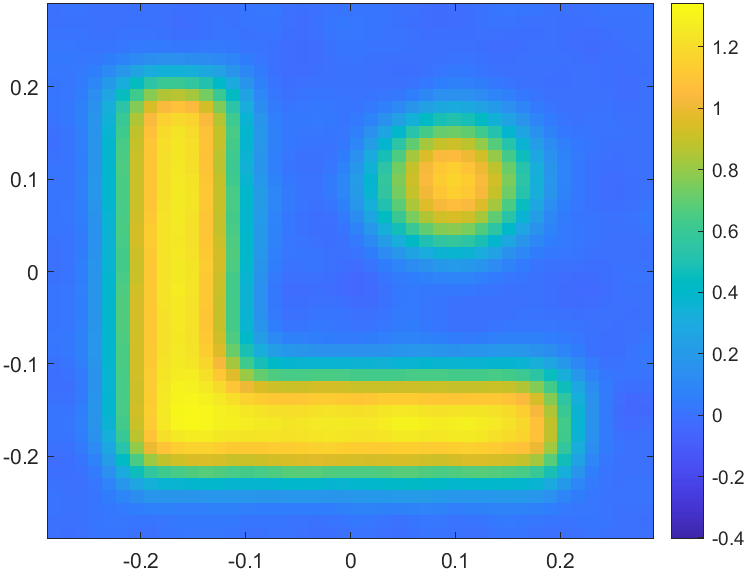}
    \caption{Example \ref{eq: example L}. \textbf{1st:} True unknown potential. \textbf{2nd:} Pointwise reconstruction using finite differences. \textbf{3rd:} Radon reconstruction using finite differences. \textbf{4th:} Radon reconstruction using regularized spectral differentiation. The jump discontinuity of the characteristic function is blurred in the reconstruction due to finite $\tau>0$ being used in the wave $v_1$, see Lemma~\ref{lemma:approx Radon}. Nevertheless, the overall geometry of the potential is captured.}
    
    \label{img: recos L}
\end{figure}

\begin{figure}[H]
    \centering
    \includegraphics[width=0.24\linewidth]{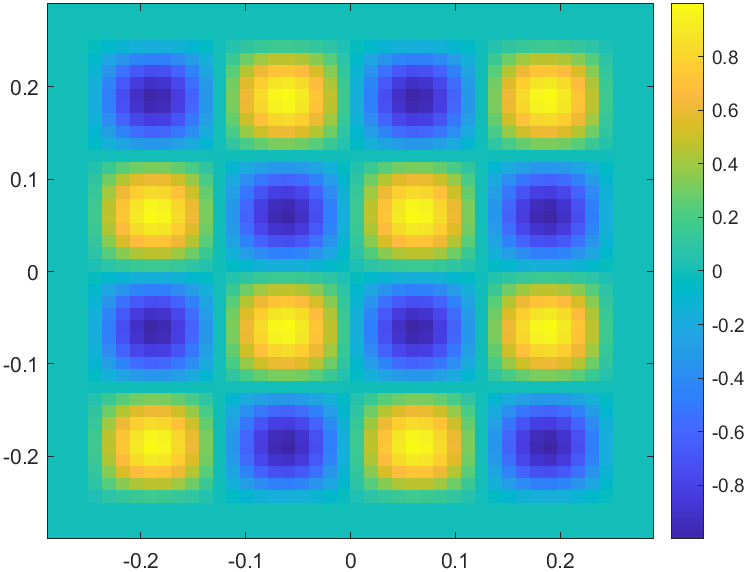}
    \includegraphics[width=0.24\linewidth]{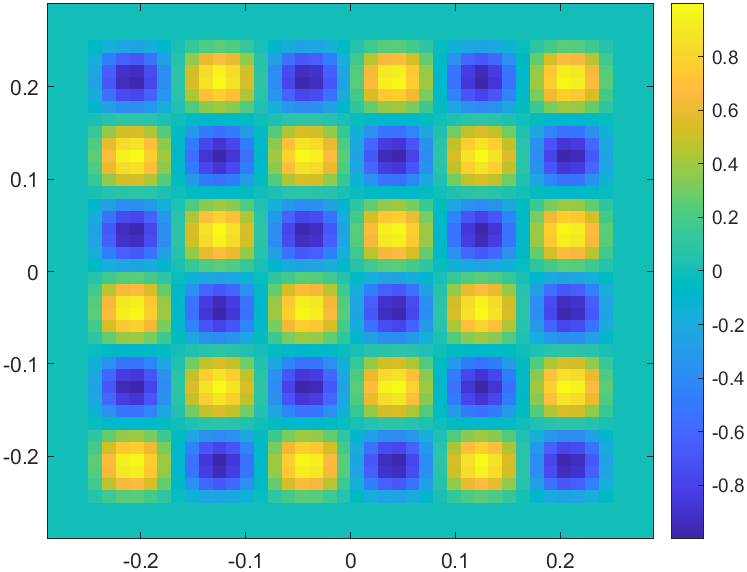}
    \includegraphics[width=0.24\linewidth]{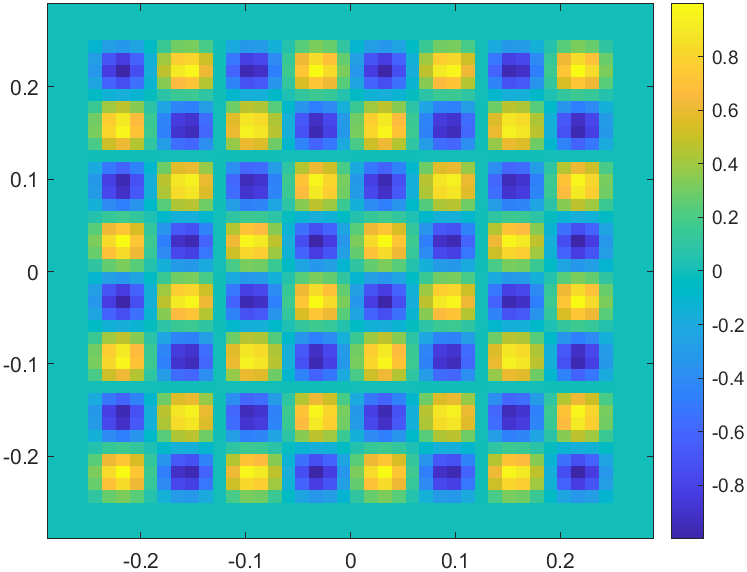}
    \includegraphics[width=0.24\linewidth]{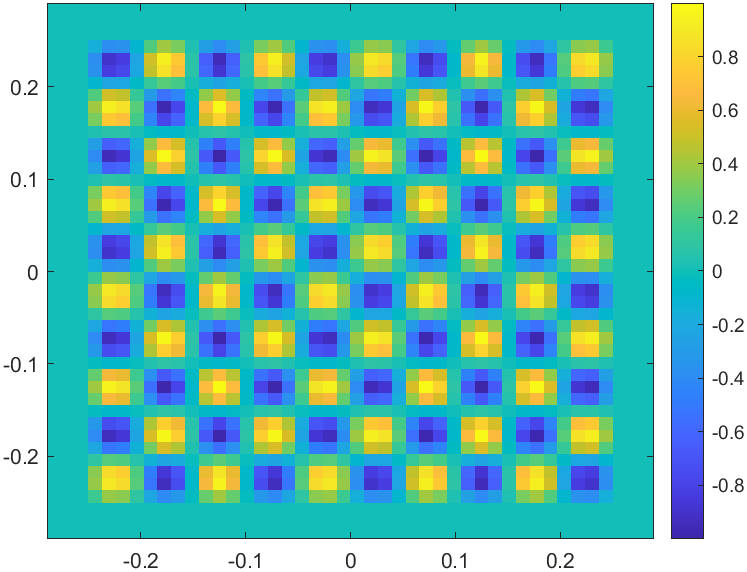}
    \includegraphics[width=0.24\linewidth]{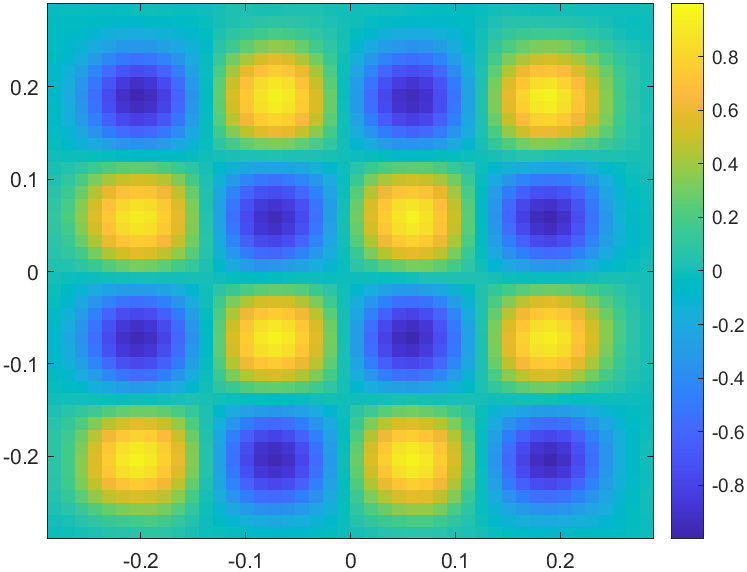}
    \includegraphics[width=0.24\linewidth]{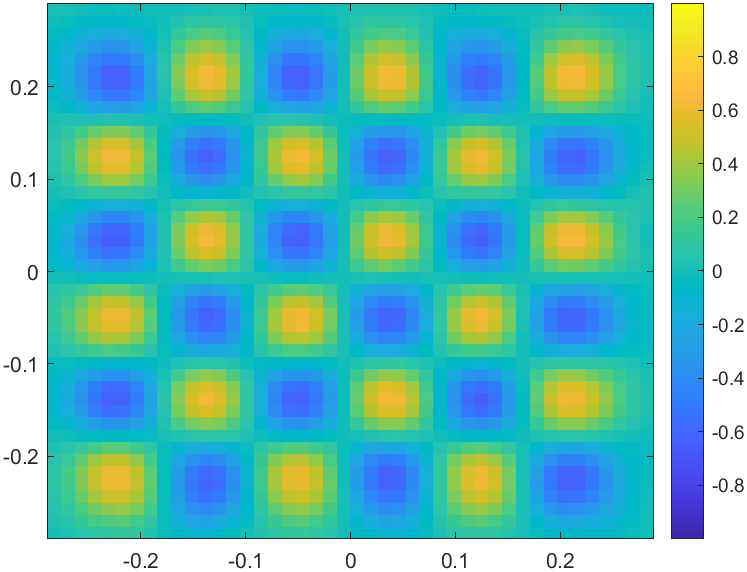}
    \includegraphics[width=0.24\linewidth]{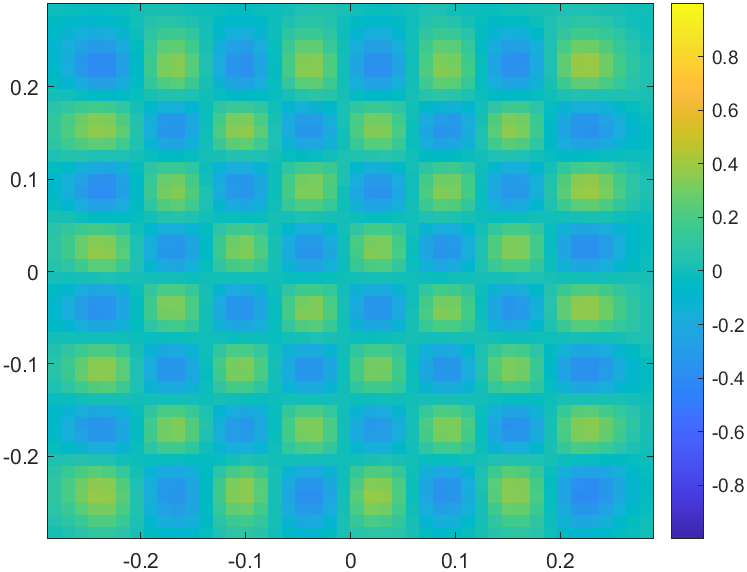}
    \includegraphics[width=0.24\linewidth]{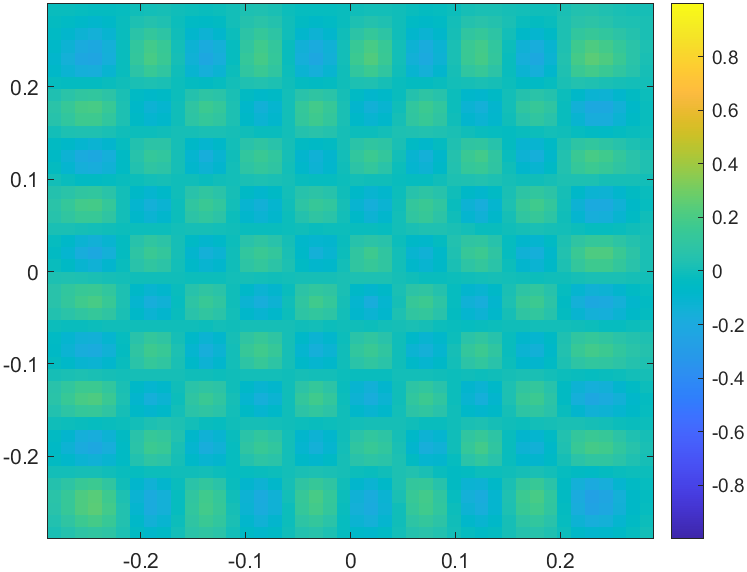}
    \caption{Example \ref{eq: example sin} with $k = 2, 3, 4, 5$. \textbf{Top row:} True unknown potentials. \textbf{Bottom row:} Reconstructions via regularized differentiation and filtered back-projection. The method reliably recovers the coarse, low-frequency structures with correct contrast and location. As the frequency $k$ in the potential $q$ increases, however, the reconstructions become progressively blurred and their amplitudes damped, with fine-scale oscillations essentially lost in the highest-frequency cases. The reconstruction parameters are kept fixed across all of the reconstructions.}
    
    \label{img: recos sin}
\end{figure}

\begin{figure}[H]
    \centering
    \includegraphics[width=0.24\linewidth]{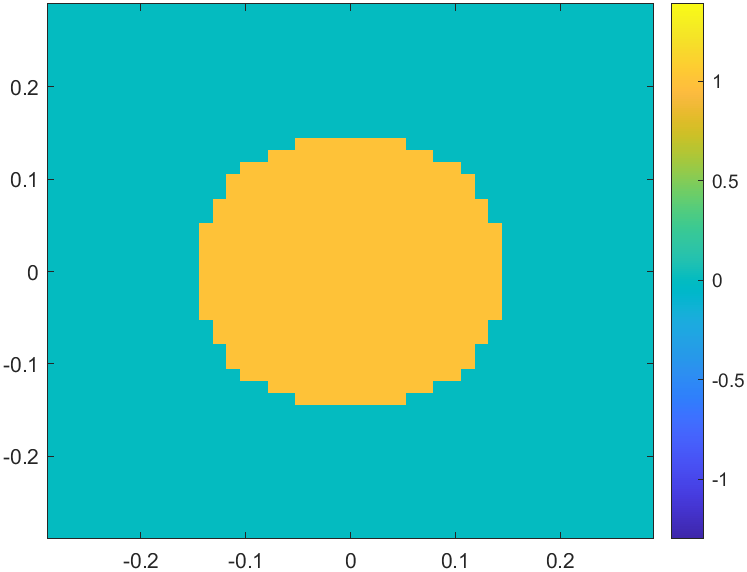}
    \includegraphics[width=0.24\linewidth]{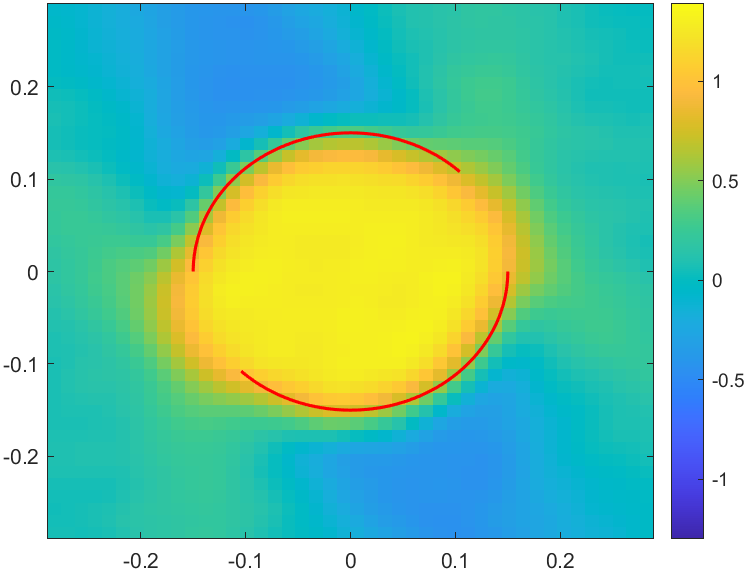}
    \includegraphics[width=0.24\linewidth]{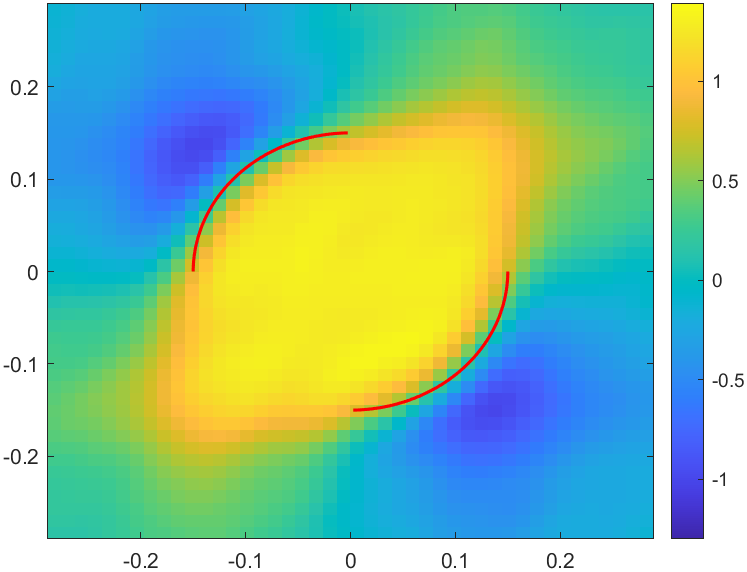}
    \includegraphics[width=0.24\linewidth]{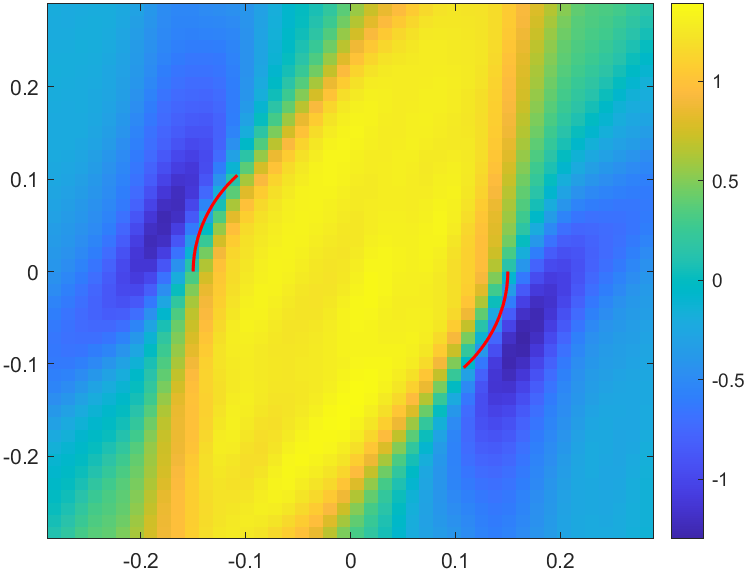}
    \caption{Example \ref{eq: example disc}, limited angle tomography. \textbf{1st:} True potential. \textbf{2nd:} Reconstruction from Radon data over $\Theta' = \{0^{\circ}, \cdots, 134^{\circ}\}$.
    \textbf{3rd:} $\Theta' = \{0^{\circ}, \cdots, 89^{\circ}\}$. \textbf{4th:} $\Theta' = \{0^{\circ}, \cdots, 44^{\circ}\}$.
    All reconstructions are obtained via regularized differentiation and filtered backprojection of the Radon transform. The locations of stable singularities (see Remark~\ref{rem: limited tomo}) are highlighted with red curves. Increasing blurring of the potential along the missing projection directions is observed as the number of incident angles is reduced.}
    
    \label{img: recos L limited angles}
\end{figure}

\begin{figure}[H]
    \centering
    \includegraphics[width=0.32\linewidth]{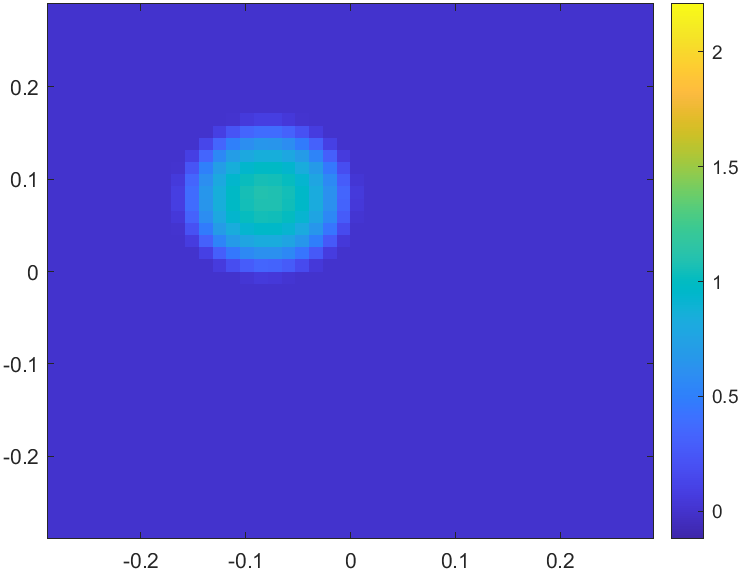}
    \includegraphics[width=0.32\linewidth]{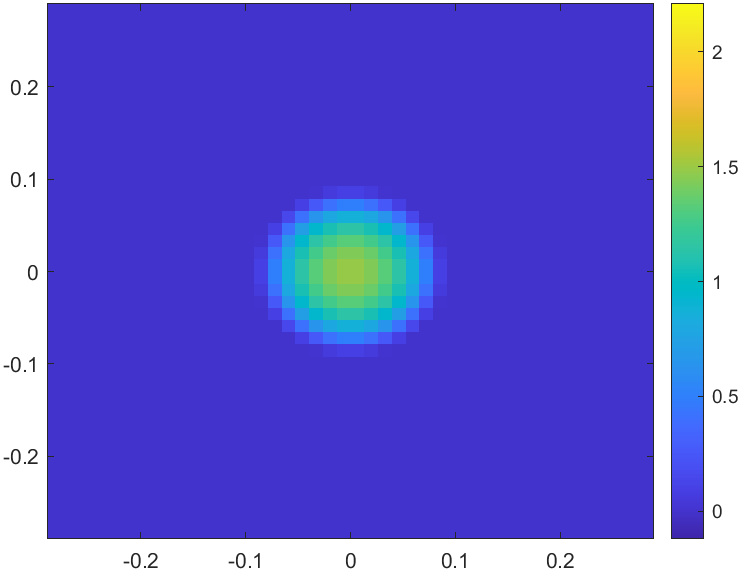}
    \includegraphics[width=0.32\linewidth]{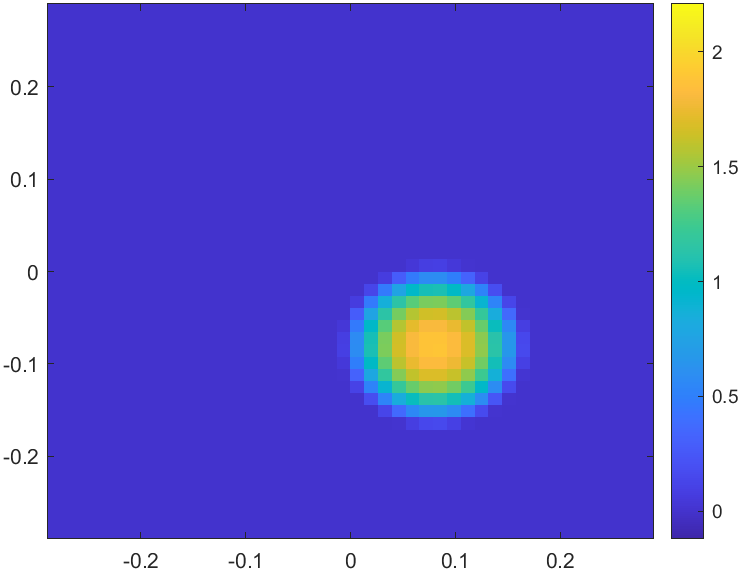}\\
    \includegraphics[width=0.32\linewidth]{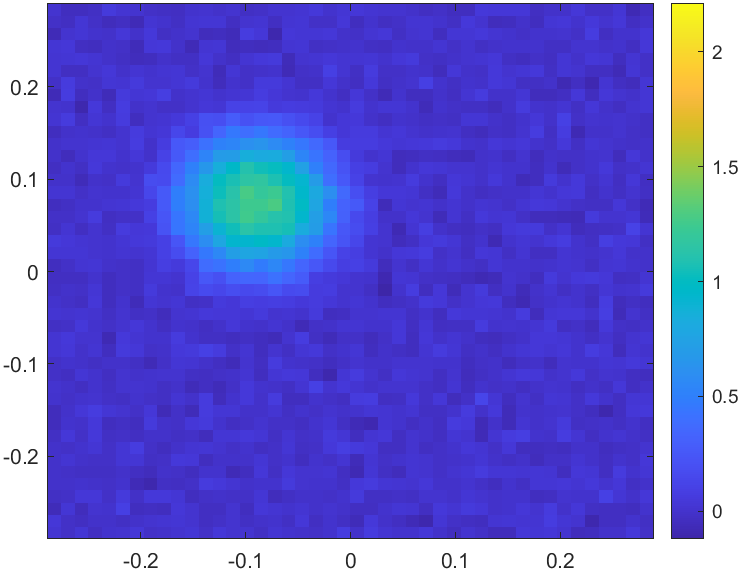}
    \includegraphics[width=0.32\linewidth]{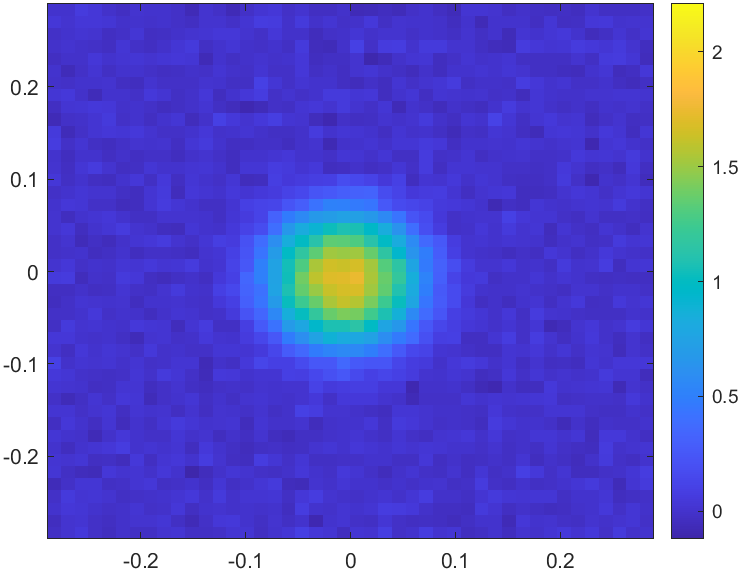}
    \includegraphics[width=0.32\linewidth]{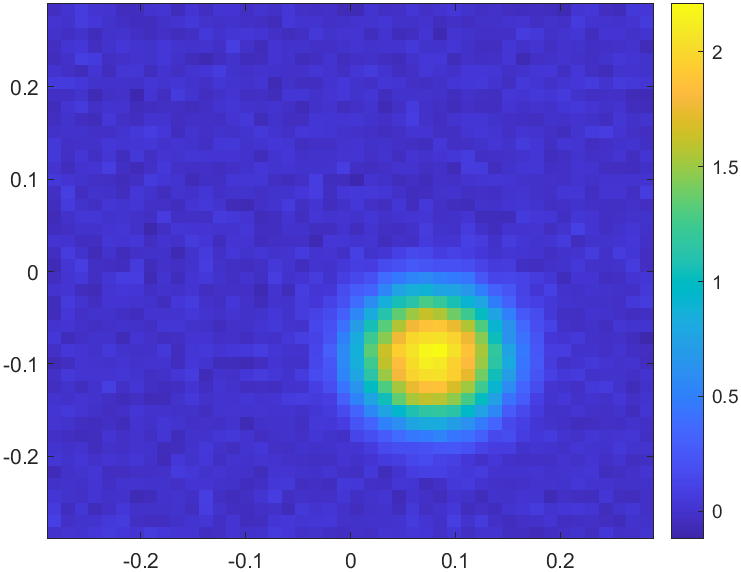}
    \caption{Example \ref{eq: example timedep}, time-dependent potential. \textbf{Top row:} Cross-sections of the true potential at times $t = 1.1, 1.5, 1.9$. \textbf{Bottom row:} Reconstructions via regularized differentiation and filtered back-projection. The potential is a smooth bump function, whose amplitude increases and position shifts linearly over time.}
    
    \label{img: timedependant}
\end{figure}

\subsection*{Acknowledgements}

This work was supported by the Research Council of Finland (Flagship of Advanced Mathematics for Sensing, Imaging and Modelling grant 359186) and Emil Aaltonen foundation.
We thank the anonymous referees for their careful reading of the manuscript and for their constructive comments and suggestions, which helped improve the paper.

\bibliography{refs} 
\bibliographystyle{abbrv}

\noindent{\footnotesize E-mail addresses:\\
Suvi Anttila: suvi.m.anttila@oulu.fi (Corresponding author)\\
Markus Harju: markus.harju@oulu.fi\\
Teemu Tyni: teemu.tyni@oulu.fi
}

\end{document}